\documentclass{article}
\title{Generic vanishing theorem for Fujiki class $\cC$}
\author{Haohao Liu}
\date{\today}
\usepackage[english]{babel}
\usepackage[T1]{fontenc} 
\usepackage{latexsym,amsmath,amssymb,amsthm,tikz-cd,graphicx}
\usepackage{mathtools}
\usepackage{tikz-cd}
\usepackage[hidelinks]{hyperref}
\usepackage{cleveref}[2012/02/15]
\crefformat{footnote}{#2\footnotemark[#1]#3}
\usetikzlibrary{patterns}

\newtheorem{thm}{Theorem}[subsection]
\newtheorem*{thm*}{Theorem}
\newtheorem{lm}[thm]{Lemma}
\newtheorem{pp}[thm]{Proposition}
\newtheorem{cor}[thm]{Corollary}
\newtheorem*{cor*}{Corollary}
\newtheorem{cjt}[thm]{Conjecture}
\newtheorem{ft}[thm]{Fact}
\theoremstyle{remark}
\newtheorem{rk}[thm]{Remark}

\theoremstyle{definition}
\newtheorem{df}[thm]{Definition}

\newtheoremstyle{example}
{\topsep} {\topsep}%
{\upshape}
{}
{\bfseries\scshape}
{.}
{1em}
{}
\theoremstyle{example}
\newtheorem{eg}[thm]{Example}
\newtheoremstyle{example_contd}
{\topsep} {\topsep}%
{\upshape}
{}
{\bfseries\scshape}
{.}
{1em}
{\thmname{#1} \thmnumber{ #2}\thmnote{#3} (continued)}

\theoremstyle{example_contd}
\newtheorem*{eg_contd}{Example}

\def\Alb{\mathrm{Alb}}
\def\an{\mathrm{an}}
\def\C{\mathbb{C}}

\def\car{\mathrm{Char}}

\def\cC{\mathcal{C}}
\def\cD{\mathcal{D}}
\def\cE{\mathcal{E}}

\def\cH{\mathcal{H}}
\def\cK{\mathcal{K}}
\def\cL{\mathcal{L}}

\def\codim{\mathrm{codim}}
\def\covdim{\mathrm{covdim}}
\def\cP{\mathcal{P}}

\def\db{\bar{\partial}}
\def\DE{\mathrm{DE}}
\def\dR{\mathrm{dR}}
\def\Ext{\mathrm{Ext}}
\def\free{\mathrm{free}}
\def\G{\mathbb{G}}
\def\GL{\mathrm{GL}}
\def\Hom{\mathrm{Hom}}

\def\Id{\mathrm{Id}}
\def\im{\mathrm{im}}
\def\Loc{\mathrm{Loc}}
\def\Mod{\mathrm{Mod}}
\def\N{\mathbb{N}}
\def\NS{\mathrm{NS}}
\def\op{\mathrm{op}}
\def\Perv{\mathrm{Perv}}
\def\Pic{\mathrm{Pic}}

\def\R{\mathbb{R}}
\def\Rep{\mathrm{Rep}}
\def\Sch{\mathrm{Sch}}
\def\Sets{\mathrm{Sets}}
\def\Supp{\mathrm{Supp}}
\def\tor{\mathrm{tor}}
\def\VB{\mathrm{VB}}
\def\Z{\mathbb{Z}}

\begin{document}
		\maketitle
	\tableofcontents
	\section{Introduction}
Recall the historical origin of generic vanishing results. In the last paragraph of \cite{Enriques1939courbes}, Enriques gave an upper bound on the dimension of the paracanonical system of curves on some class of algebraic surfaces. However, in \cite[p.354]{enriques1949superficie} he pointed out a mistake in the proof of his result as well as a similar theorem by Severi \cite{severi1942teoria}. Catanese \cite[p.103]{catanese1983moduli} posed Conjecture \ref{cjt:Catanese}.
\begin{cjt}\label{cjt:Catanese}For a smooth projective surface $S/\C$ without irrational pencils,   the dimension of the paracanonical system $\{K_S\}$ is at most the geometric genus $p_g(S)$.\end{cjt}
In 1987, Green and Lazarsfeld \cite[Theorem 4.2]{green1987deformation} provided a positive answer to Conjecture \ref{cjt:Catanese}. Its proof uses a result (\cite[Prop.~4.1]{green1987deformation}) of  generic vanishing type. 

As is explained in \cite[pp.619--620]{ueno1983classification}, the dimension of $\{K_S\}$  in Conjecture \ref{cjt:Catanese} is related to Conjecture \ref{cjt:Beauville}, which is also of generic vanishing type.
\begin{cjt}[{\cite[Problem 8, p.620]{ueno1983classification}}]\label{cjt:Beauville}
Let $X$ be a projective manifold and $\alpha:X\to \Alb(X)$ be an Albanese morphism. If $\dim \alpha(X)>1$, then $H^1(X,L)=0$ for generic $L\in\Pic^0(X)$. 
\end{cjt}Green and Lazarsfeld \cite{green1987deformation} proved a strengthening  of Conjecture \ref{cjt:Beauville}.  Since then, the theory of generic vanishing results has been very  much investigated and numerous authors have contributed to its development, so the   overview in Section \ref{sec:known} is by no means complete.

For a finitely generated $\Z$-module $H$, let $H_{\tor}$ be the  submodule of $H$ comprised of torsion elements and $H_{\free}:=H/H_{\tor}$. Let $F\to X$ be a  (holomorphic)  vector bundle\footnote{We use the words ``locally free sheaf" and ``vector bundle" interchangeably.} on a   complex manifold. The dimension of a complex space always means the complex dimension. For any three integers $p,q,m\ge0$, the corresponding jumping locus  is defined as \[S^{p,q}_m(X,F):=\{L\in \Pic^0(X):h^q(X,\Omega_X^p\otimes_{O_X}L\otimes_{O_X}F)\ge m\}.\] For simplicity, $p$ (resp. $m$, resp. $F$) is omitted when $p=0$ (resp. $m=1$, resp. $F=O_X$). Roughly speaking, generic vanishing results show that these loci are small (in some sense) and study their  structure when $F$ is flat unitary (in the sense of Definition \ref{df:flatuni}). 
\subsection{Known results}\label{sec:known}
Let $X$ be a connected compact Kähler manifold,  $\alpha:X\to \Alb(X)$ be the Albanese map associated with some base point and $F\to X$ be a flat unitary vector bundle.  Each locus $S^{p,q}_m(X,F)$ is an analytic subset of the complex torus $\Pic^0(X)$ (see the proof of Theorem \ref{thm:main} \ref{it:locusana}) and ``generic"  means outside a strict analytic subset.  In the literature, generic vanishing results  concerning $S^q(X,F)$ (resp. $S^{p,q}(X,F)$) are usually called of Kodaira type (resp. Nakano type). Such results typically involve the following invariants:\begin{itemize}
   	\item $\dim\alpha(X)$;
   	\item $w(X):=\max\{\mathrm{codim}(Z(\eta),X):0\neq \eta\in H^0(X,\Omega_X^1)\}$, where $Z(\eta)$ denotes the zero-locus of the $1$-form $\eta$;
   	\item the defect of semismallness $r(\alpha)$ of $\alpha$ (Section \ref{sec:equivdf}).
   \end{itemize}
   
Using deformation theory of  cohomology groups, Green and Lazarsfeld \cite[Remarks (1), p.401]{green1987deformation} proves  Fact \ref{ft:GL}, which is of Kodaira type and implies Conjecture \ref{cjt:Beauville}.
\begin{ft}\label{ft:GL}
For every integer $k\ge0$, one has \[\codim_{\Pic^0(X)}(S^k(X,F))\ge \dim \alpha(X)-k.\] In particular, if $k<\dim \alpha(X)$, then $H^k(X,F\otimes_{O_X}L)=0$ for    a generic line bundle $L\in \Pic^0(X)$. 
\end{ft}
Green and Lazarsfeld also give a Nakano-type generic vanishing theorem. 
\begin{ft}[{\cite[Remarks (1), p.404]{green1987deformation}}]\label{ft:GLN} For any integers $i,j\ge0$ with $i+j<w(X)$, one has $S^{i,j}(X, F)\neq \Pic^0(X)$.
\end{ft}
In another direction, there are known results concerning the structure of the jumping loci.
\begin{ft}[{\cite[Thm.~0.1 (1)]{green1991higher}}]\label{ft:Skm}For any two integers $k,m\ge0$, the subset $S^k_m(X)$
is a finite union of translates of subtori of $\Pic^0(X)$.
\end{ft}

Beauville and Catanese  conjectured that for every integer $q\ge0$, $S^q(X)$ is a finite union of \emph{torsion} translates of subtori of $\Pic^0(X)$ (\cite[Problem 1.25]{catanese1991moduli} and \cite[p.1]{beauville1992annulation}). When $X$ is a projective manifold, this conjecture is proved by Simpson \cite[Sec. 5]{simpson1993subspaces}.
\begin{ft}[Simpson]
If $X$ is furthermore projective, then for any two integers $k,m\ge0$, the locus $S^k_m(X)$  is a finite union of \emph{torsion} translates of subtori of $\Pic^0(X)$.
\end{ft}
Some arguments of \cite{simpson1993subspaces} are of arithmetic nature, so they do not apply to the Kähler case. Campana \cite[Sec.~1.5.2]{campana2001ensembles} provided a partial answer for not only Kähler manifolds but also  for Fujiki class $\cC$ (Definition \ref{df:Fujiki}).

Later on, Wang \cite[Cor.~1.4]{wang2016torsion} answered affirmatively Beauville and Catanese's conjecture in full generality.
\begin{ft}[Wang]
For any three integers $p,q,m\ge0$, the subset  $S^{p,q}_m(X)$   of $\Pic^0(X)$ is a finite union of \emph{torsion} translates of subtori. 
\end{ft}
 Hacon \cite[Cor.~4.2]{hacon2004derived} uses Fourier-Mukai transforms of coherent modules on complex abelian varieties to recover Fact \ref{ft:GL} when $X$ is a projective manifold. This    algebraic viewpoint   sheds new insight on this topic.  Similarly, as a byproduct of the theory on convolution of perverse sheaves on  abelian varieties, Krämer and Weissauer obtain a  Nakano-type generic vanishing theorem. The proof of \cite[Thm.~3.1]{kramer2015vanish} gives Fact \ref{ft:KWvanish}.
\begin{ft}\label{ft:KWvanish}  If furthermore the Albanese torus $\Alb(X)$ is algebraic, then for any two integers $p,q\ge0$ with $p+q<\dim X-r(\alpha)$, the locus $S^{p,q}(X,F)$ is contained in a finite union of translates of strict subtori of $\Pic^0(X)$.
\end{ft}

Around the  same time, by different methods Popa and Schnell \cite[Thm.~1.2]{popa2013generic} obtained precise codimension bounds.\begin{ft}
If furthermore $X$ is a projective manifold, then \[\codim_{\Pic^0(X)}(S^{p,q}(X))\ge |p+q-\dim X|-r(\alpha)\] for any two integers $p,q\ge0$. Moreover, for every $X$ there exist $p$ and $q$ for which the inequality becomes an equality.\end{ft}
 \subsection{The main result and a sketch}
Even though not necessarily Kähler, a complex smooth proper algebraic variety\footnote{An algebraic variety means an integral scheme of finite type and separated over a field.} also admits  Hodge theory (\cite[Prop.~5.3]{deligne1968theoreme}, \cite[Thm.~3.2.5]{deligne1971theorie}).
 It is  natural to ask if generic vanishing results also hold for such varieties. The aim of this note is to show that generic vanishing  result is not only true for Kähler manifolds, but also for  complex manifolds in Fujiki class $\cC$. This class contains compact Kähler manifolds as well as smooth proper algebraic varieties. 

Here is the main result that is of Nakano type.
\begin{thm*}[Theorem \ref{thm:main}]
Let $X$ be an $n$-dimensional complex manifold in Fujiki class $\cC$ with an Albanese morphism\footnote{reviewed in (\ref{eq:Albmor})} $\alpha:X\to\Alb(X)$, and let $F$ be a flat unitary  vector bundle on $X$. Then for any two integers $p,q\ge0$ with $p+q<n-r(\alpha)$, the locus $S^{p,q}(X,F)$ is a strict analytic subset of the complex torus $\Pic^0(X)$.
\end{thm*}
For  smooth proper algebraic varieties,  the following  finer result follows from Corollary \ref{cor:agmain} and Lemma \ref{lm:thintrans}. It is not immediate from previously known generic vanishing results.
\begin{cor}
	Let	$X/\C$ be an $n$-dimensional smooth proper algebraic variety with an algebraic Albanese morphism\footnote{\label{foot:algAlb}recalled in Section \ref{sec:Moishezon}} $\alpha:X\to \Alb(X)$. Let $\cL$ be a unitary local system on the analytification $X^{\an}$, and let $F=\cL\otimes_{\C}O_{X^{\an}}$ be the corresponding holomorphic vector bundle. Then,  for any two  integers $p,q\ge0$ with $p+q<n-r(\alpha)$, the subset $S^{p,q}(X,F)$ is contained in a finite union of  translates (\emph{torsion} translates if $\cL$ is semisimple of geometric origin\footnote{\label{foot:geoori}in the sense of \cite[p.163]{beilinson2018faisceaux}}) of strict abelian subvarieties of the Picard variety\cref{foot:algAlb} $\Pic^0_{X/\C}$.\end{cor}

 Here is  the outline of the proof of Theorem \ref{thm:main}. By the Riemann-Hilbert correspondence restricted to  unitary objects, we pass from flat unitary vector bundles to unitary local systems. The corresponding cohomology groups are related by Hodge decomposition (Fact \ref{ft:FujikiHodge}). In this way, the initial generic vanishing problem for a flat unitary vector bundle twisted by line bundles is reduced to a generic vanishing problem for a unitary local system twisted by rank $1$ local systems.

By   pushing forward along the Albanese map,  the problem about the local system on a manifold in Fujiki class $\cC$ is converted to a problem about a complex of sheaves on a complex torus. The last problem is solved by Krämer and Weissauer \cite{kramer2015vanish} for perverse sheaves (on complex abelian varieties) and by the subsequent generalization (to all complex tori) due to Bhatt, Schnell and Scholze \cite{bhatt2018vanishing}.

This text is organized as follows. Sections \ref{sec:RH} reviews the unitary Riemann-Hilbert correspondence. Section \ref{sec:Jac} and \ref{sec:Alb} construct  the Jacobian and the Albanese map for regular manifolds, relaxing the usual Kähler condition. Several definitions of defect of semismallness are proved to be equivalent in Section \ref{sec:defect}. The work of Krämer and Weissauer on generic vanishing for perverse sheaves is recalled in Section \ref{sec:KW}. Finally in Section \ref{sec:main}, the previous results are applied to prove the main result, Theorem \ref{thm:main},  for Fujiki class $\cC$.

\section*{Acknowledgment}First, I am grateful to my supervisor Anna Cadoret. It was her constant encouragement that has made this note produced. I must thank her for carefully reading the manuscript multiple times and making numerous suggestions that have substantially improved the exposition. Many times, she indicated me how to improve the results by removing unnecessary conditions. My gratitude goes to Professor Botong Wang for answering my question on his work. I am also indebted to other friends including: Chenyu Bai, Junbang Liu, Long Liu, Qiaochu Ma, Miao Song, Mingchen Xia, Lu Yu, Hui Zhang and Junsheng Zhang.
All remaining errors and omissions are mine.
\section{Riemann-Hilbert correspondence}\label{sec:RH}
In Section \ref{sec:RH}, we review how the classical Riemann-Hilbert correspondence restricts to an equivalence between unitary local systems and flat unitary  vector bundles on  complex manifolds. The reason to introduce this restricted equivalence is that  unitary local systems on  manifolds in Fujiki class $\cC$ admit Hodge decomposition (Fact \ref{ft:FujikiHodge}).
\subsection{Unitary local systems}\label{sec:loc}
Let $X$ be a path-connected,  locally path-connected and locally simply connected topological space with a base point $x_0\in X$. Let $\Loc(X)$ be the category of local systems (of finite dimensional $\C$-vector spaces) on $X$. Let $\pi_1(X,x_0)$ be the fundamental group of $X$ at $x_0$ and $\Rep_{\C}(\pi_1(X,x_0))$ be the category  of its finite dimensional complex representations. By \cite[Cor.~1.4, p.4]{deligne2006equations}, the  functor  taking the stalk at $x_0$ gives rise to an equivalence \begin{equation}\label{eq:locmon}\Loc(X)\to \Rep_{\C}(\pi_1(X,x_0))\end{equation}  compatible with tensor products. The image under (\ref{eq:locmon}) of a local system on $X$ is called the corresponding monodromy representation.

Let $ \Rep^u_{\C}(\pi_1(X,x_0))\subset \Rep_{\C}(\pi_1(X,x_0))$ be the full subcategory of unitary representations. That means representations $\rho:\pi_1(X,x_0)\to \GL(V)$ satisfying the following equivalent\footnote{since every compact subgroup of $\GL_r(\C)$ can be conjugated into the unitary subgroup $U_r(\C)$} conditions: 
\begin{enumerate}
\item The closure of $\rho(\pi_1(X,x_0))$ inside $\GL(V)$ is compact;
\item There is a hermitian inner product $h:V\otimes_{\C}\bar{V}\to \C$ such that $\rho(\pi_1(X,x_0))$ is contained in the corresponding unitary group $U(V,h)$. 
\end{enumerate} 
Let $\Loc^u(X)$ be the  full subcategory of $ \Loc(X)$ corresponding to $\Rep^u_{\C}(\pi_1(X,x_0))$ \textit{via} the equivalence (\ref{eq:locmon}).
 Its objects   are called unitary local systems on $X$. Every unitary local system is semisimple, since every unitary representation  is so. 
\subsection{Flat unitary bundles}
 Let  $E\to X$ be a holomorphic vector bundle on a complex manifold with a hermitian metric $h$.  By \cite[Prop.~4.2.14]{huybrechts2005complex}, there exists a unique hermitian connection $\nabla_h$ that is compatible with the holomorphic structure (in the sense of \cite[Def.~4.2.12]{huybrechts2005complex}, \textit{i.e}., $\nabla^{0,1}=\db^E$), which is called the Chern connection of $(E,h)$. The corresponding curvature form, called the Chern curvature, is an  $\mathrm{End}(E,h)$-valued $(1,1)$-form, (see, \textit{e.g.}, \cite[Prop.~4.3.8 iii)]{huybrechts2005complex}). 

For every integer $k\ge0$ (resp. any two integers $i,j\ge0$), let $A^k_X$ (resp. $A^{i,j}_X$) be 
the sheaf of \emph{smooth} $k$ (resp. $(i,j)$) forms on $X$. Then there is a direct sum decomposition $A^k_X=\oplus_{i+j=k}A^{i,j}_X$.  In general, a (smooth) \emph{flat} connection $\nabla$ on $E$ that is compatible with the holomorphic structure  needs not to be a holomorphic connection (in the sense of \cite[Def.~4.2.17]{huybrechts2005complex}).

\begin{lm}\label{lm:MSE4583322}
Let $E\to X$ be a holomorphic vector bundle with a  \emph{flat} connection $\nabla:E\to E\otimes_{A_X^0}A_X^1$. If $\nabla$ is compatible with the holomorphic structure, then $\nabla$ is a holomorphic connection.
\end{lm}
\begin{proof}
 Take a local holomorphic frame $\{e_1,\dots,e_r\}$ of $E$, and denote the corresponding local smooth connection matrix  $1$-form by $\Omega$. As $\nabla^{0,1}=\bar{\partial}^E$, one has $\Omega^{0,1}=0$. By flatness, $d\Omega+\Omega\wedge \Omega=0$. Taking the $(1,1)$ part of it, one gets $\bar{\partial}\Omega=0$, \textit{i.e}., $\Omega$ is a holomorphic form. This shows that $\nabla$ is holomorphic.
\end{proof}
Let $\Mod(O_X)$ be the category of $O_X$-modules, and let $\VB(X)\subset \Mod(O_X)$ be the full subcategory of  finite locally free $O_X$-modules. Let $\DE(X)$ be the category of holomorphic vector bundles with a flat \emph{holomorphic} connection. Forgetting the connection gives a  functor $\DE(X)\to \VB(X)$. Let $\DE^u(X)\subset \DE(X)$ be the full subcategory comprised of objects $(F,\nabla)$ such that there exists a hermitian metric on $F$ whose Chern connection is $\nabla$. 
\begin{df}\label{df:flatuni}An object in the essential image of $\DE^u(X)$ under the forgetful functor $\DE(X)\to \VB(X)$ is called a \emph{flat unitary} vector bundle on $X$.\end{df} From \cite[Eg.~4.2.15]{huybrechts2005complex}, the trivial line bundle $O_X$ is flat unitary. By  Lemma \ref{lm:MSE4583322}, a holomorphic vector bundle is flat unitary  if and only if it admits a hermitian metric whose Chern connection is flat.
\subsection{An equivalence}
Let $X$ be a connected complex manifold. By the  Riemann-Hilbert correspondence (\cite[Thm.~2.17, p.12]{deligne2006equations}), the pair of functors \begin{gather}\Loc(X)\to \DE(X),\quad  \cL\mapsto (\cL\otimes_{\C}O_X,\Id_{\cL}\otimes d);\label{eq:LoctoDE}\\
	\DE(X)\to \Loc(X),\quad (E,\nabla)\mapsto \ker(\nabla)\label{eq:DEtoLoc}
\end{gather} forms an equivalence of categories. It is compatible with tensor products and preserves the rank.

\begin{thm}[Unitary Riemann-Hilbert correspondence]\label{thm:unitRH}
	The equivalence (\ref{eq:LoctoDE}), (\ref{eq:DEtoLoc}) restricts to an equivalence between $\Loc^u(X)$ and $\DE^u(X)$. 
\end{thm}
\begin{proof}
First, we prove that the functor (\ref{eq:LoctoDE}) sends $\Loc^u(X)$ to $\DE^u(X)$. Consider a unitary local system $\cL$ on $X$. Since the corresponding monodromy representation is unitary, we may choose a hermitian inner product $h_{x_0}$ on the stalk $\cL_{x_0}$ such that  the representation  factors through $U(\cL_{x_0},h_{x_0})$. For any $x\in X$, choose a path $\gamma$ from $x_0$ to $x$ and propagate $h_{x_0}$ along this curve, \textit{i.e}., using the linear isomorphism $\gamma_*:\cL_{x_0}\to \cL_x$ induced by $\gamma$, we translate $h_{x_0}$ to  a hermitian inner product $h_x$ of $\cL_x$. This $h_x$ is independent of the choice of  $\gamma$ by assumption. Hence a positive definite hermitian form $h$ on $\cL$ that is invariant under the monodromy action.  Then  $h$ extends naturally to a (smooth) hermitian metric $h'$ on the associated holomorphic vector bundle $\cL\otimes_{\C}O_X$ on $X$ and the corresponding flat holomorphic connection
$\Id_{\cL}\otimes d$ is a  hermitian connection.  Therefore, $\Id_{\cL}\otimes d$ is the Chern connection of $(\cL\otimes_{\C}O_X,h')$ and $(\cL\otimes_{\C}O_X,\Id_{\cL}\otimes d)\in \DE^u(X)$.

Conversely, we prove that the functor (\ref{eq:DEtoLoc}) sends $\DE^u(X)$ to $\Loc^u(X)$. Consider a holomorphic hermitian vector bundle $(E,h)$ on $X$  whose Chern   connection $\nabla_h$ is  flat. Around every point we can find a local $\nabla_h$-horizontal holomorphic frame $\{e_1,\dots,e_r\}$ of $E$. For any $1\le i,j\le r$, since the connection $\nabla_h$ is compatible with $h$, we have  \[d[h(e_i,e_j)]=h(\nabla_h e_i,e_j)+h(e_i,\nabla_h e_j)=0.\] Therefore, the local function $h(e_i,e_j)$ is locally constant and the parallel transport along every closed path on $X$ preserves the hermitian inner products on the fibers of $E$. The sheaf $\ker(\nabla_h)$  of horizontal sections of $E$ forms a  local system on $X$, whose stalks are exactly the fibers of $E$. Thus, it admits a monodromy-invariant positive definite hermitian form and is consequently unitary.\footnote{The definition of unitary local system  in \cite[p.152]{timmerscheidt1987mixed}  seems  to forget  this invariance.}
\end{proof}
\section{Hodge theory and Jacobian}\label{sec:Jac}
 In  Section \ref{sec:Jac}, we review the definition of  Jacobian  and show that for every complex manifold admitting Hodge theory (Definition \ref{df:regular}),   its Jacobian has nice expected properties. 

\subsection{Regular manifolds}\label{sec:regular}
Let $X$ be a   complex manifold. Let $d:A^{\bullet}_X\to A^{\bullet+1}_X$ be the exterior derivative. Then $d=\partial+\db$, where $\partial:A^{\bullet,\bullet}_X\to A^{\bullet+1,\bullet}_X$ and $\db:A^{\bullet,\bullet}_X\to A^{\bullet,\bullet+1}_X$ are the $(1,0)$ and $(0,1)$ part of $d$ respectively.
For every $\cE\in \Loc^u(X)$, every integer $k\ge0$, define a decreasing filtration of $A^k_X\otimes_{\C}\cE$ by \begin{equation}\label{eq:filtration}F^p=F^p(A^k_X\otimes_{\C}\cE):=\oplus_{i\ge p}A^{i,k-i}_X\otimes_{\C}\cE.
\end{equation} Then $(d\otimes\Id_{\cE})(F^p)\subset F^p$.  Therefore, this filtration induces a spectral sequence, called the Frölicher spectral sequence: \begin{equation}\label{eq:Frspe}
E_1^{p,q}=H^q(X,\Omega_X^p\otimes_{\C}\cE)\Rightarrow H^{p+q}(X,\cE),
\end{equation}
where the differential $d_1^{p,q}:E_1^{p,q}\to E_1^{p+1,q}$ is induced by the operator $\partial:A^{p,q}_X\to A^{p+1,q}_X$ on $X$. It is the classical notion in \cite[Sec. 8.3.3]{voisin2002hodge} when $\cE$ is the constant sheaf  $\C_X$.

 Although the  Hodge theory for the first cohomology groups $H^1$ suffices for most properties of the Jacobian  and the Albanese,  in the sequel we mainly work with manifolds admitting Hodge theory in all degrees. Such manifolds are called ``regular" for convenience.
\begin{df}[Regular manifold, {\cite[5.21 (2)]{deligne1975real}}]\label{df:regular} Assume that $X$ is compact. Let $\cE\in \Loc^u(X)$. 
If the following conditions are satisfied: \begin{enumerate}
	\item The corresponding spectral sequence (\ref{eq:Frspe}) degenerates at page $E_1$;
	\item For every integer $k\ge0$, the filtration induced by $F^{\bullet}(A^{\bullet}_X\otimes_{\C}\cE)$ on $H^k(X,\cE)$ gives a complex Hodge  structure of weight $k$, in particular a Hodge decomposition \begin{equation}\label{eq:decompE}
		H^k(X,\cE)=\oplus_{p+q=k}H^q(X,\Omega_X^p\otimes_{\C}\cE);
	\end{equation} 
	\item 	For any integers $p,q\ge0$, the conjugation map induces a $\C$-anti-linear isomorphism \[H^q(X,\Omega_X^p\otimes_{\C} \cE)\to H^p(X,\Omega_X^q\otimes_{\C}\cE^{\vee}),\]where  $\cE^{\vee}=\cH om(\cE,\C_X)$ is the dual local system.
\end{enumerate} Then $X$ is called $\cE$-regular (and simply \emph{regular} when $\cE=\C_X$).
\end{df}
  For instance, classical Hodge theory asserts that compact Kähler manifolds are regular (see \textit{e.g.}, \cite[Sec. 6.1.3]{voisin2002hodge}).      Because of Fact \ref{ft:deedee}, regular manifolds are also called $\partial\db$-manifolds.
\begin{ft}[$\partial\db$-lemma, {\cite[5.14, 5.21]{deligne1975real}, \cite[Prop.~3.4]{varouchas1986proprietes}, \cite[Cor.~3.2.10]{huybrechts2005complex}}]
	\label{ft:deedee}Assume that $X$ is compact. Then $X$ is regular  if and only if for every	 $d$-closed smooth $(p,q)$-form $\eta$ on $X$, the following conditions are equivalent:
\begin{enumerate}
\item $\eta$ is $d$-exact;
\item $\eta$ is $\partial$-exact;
\item $\eta$ is $\db$-exact;
\item $\eta$ is $\partial\db$-exact.
\end{enumerate}
If the above conditions hold and $\eta$ is real, then there is a real  smooth $(p-1,q-1)$-form $\rho$ on $X$ with $\eta=i\partial\db \rho$.
\end{ft} 
\begin{rk}
For Fact \ref{ft:deedee}, it is important that the decomposition (\ref{eq:decompE}) is induced by the filtration (\ref{eq:filtration}). In fact, \cite[Prop.~4.3]{ceballos2016invariant} constructs  a non-regular compact complex manifold $X$ of dimension $3$, such that  the spectral sequence (\ref{eq:Frspe}) for $\cE=\C_X$ degenerates at page $E_1$, with numerical Hodge symmetry $h^{p,q}(X)=h^{q,p}(X)$ for any two integers $p,q\ge0$. In this case, there is a non canonical decomposition of the form  (\ref{eq:decompE}).
\end{rk}
For the rest of  Section \ref{sec:regular}, \textbf{we assume that $X$ is a regular manifold}. 	For every integer $k\ge0$ (resp. any two integers $p,q\ge0$), the space of global $\partial$-closed, $\db$-closed smooth $k$ (resp. $(p,q)$) forms on $X$ is denoted by $Z^k(X)$ (resp. $Z^{p,q}(X)$). For any two integers $p,q\ge0$, the Dolbeault cohomology group $H^q(X,\Omega_X^p)$ is denoted by $H^{p,q}(X)$. 
 \begin{cor}\label{cor:decomp}   For any  integers $p,q\ge0$ and $k:=p+q$, there is a canonical commutative diagram \begin{center}
\begin{tikzcd}
	{Z^{p,q}(X)} \arrow[r, hook] \arrow[d] & Z^k(X) \arrow[d] \\
	{H^{p,q}(X)} \arrow[r,"\iota^{p,q}"]           & H^k(X),          
\end{tikzcd}
 	\end{center}
 where the first row is the natural inclusion and each vertical map is surjective.
Moreover,
\begin{equation}\label{eq:Hodge}
H^k(X,\C)=\oplus_{p+q=k}\im(\iota^{p,q}),
\end{equation} where each $\im(\iota^{p,q})$ can be identified with $H^{p,q}(X)$.
The complex conjugation map $Z^{p,q}(X)\to Z^{q,p}(X)$ descends to a $\C$-antilinear isomorphism $H^{p,q}(X)\to H^{q,p}(X)$ (Hodge symmetry). 
 \end{cor}

\begin{proof}
	For each $\db$-closed $(p,q)$-form $\eta$ on $X$, $\partial \eta$ is a  $d$-closed, $\partial$-exact $(p+1,q)$-form. By Fact \ref{ft:deedee}, there is a $(p,q-1)$-form $\rho$ on $X$ with $\partial \eta+\partial\db\rho=0$, then the $(p,q)$-form $\eta+\db\rho$ is in  $Z^{p,q}(X)$. Therefore, the map taking Dolbeault cohomology class $Z^{p,q}(X)\to H^{p,q}(X)$ is  surjective. 
	
Note that $\eta+\db\rho$ is  $d$-closed. Its de Rham cohomology class is independent of the choice of $\rho$. Indeed,	if $\rho'$ is another $(p,q-1)$-form with $\eta+\db\rho'$ also $d$-closed, then $\db(\rho-\rho')$ is $d$-closed and $\db$-exact. By Fact \ref{ft:deedee}, it is $d$-exact.

Thus the map \[\iota^{p,q}:H^q(X,\Omega_X^p)\to H^{p+q}_{\dR}(X,\C),\quad [\eta]\to [\eta+\db\rho]\] is a well-defined  $\C$-linear map. By a third application of Fact \ref{ft:deedee}, the map $\iota^{p,q}$ is injective. Thus, $H^{p,q}$ is identified with $\im(\iota^{p,q})$.

	We claim that the sum $\sum_{p+q=k}\im(\iota^{p,q})$ is direct. In fact, if $\alpha^{p,q}\in Z^{p,q}(X)$  for each pair $(p,q)$ with $p+q=k$ and the de Rham class of $\sum_{p+q=k}\alpha^{p,q}$ is $0$ in $H^k_{\dR}(X,\C)$, then there is a $(k-1)$-form $\beta$ on $X$ with $d\beta=\sum_{p+q=k}\alpha^{p,q}$. Thus, \[\alpha^{p,q}=\partial(\beta^{p-1,q})+\db(\beta^{p,q-1}).\] The $\partial$-exact form $\partial(\beta^{p-1,q})$ is thereby $\db$-closed, so $d$-closed. By Fact \ref{ft:deedee} again, $\partial(\beta^{p-1,q})$ is $\db$-exact, hence $[\alpha^{p,q}]=0$ in $H^{p,q}(X)$ for every $(p,q)$. The claim is proved. 

By assumption, \[\dim_{\C}H^k(X,\C)=\sum_{p+q=k}\dim_{\C}H^{p,q}(X),\] hence the decomposition (\ref{eq:Hodge}). In particular,
the map taking de Rham cohomology class $Z^k(X)\to H^k(X,\C)$ is surjective. 
The complex conjugate of $Z^{p,q}(X)$ is exactly $Z^{q,p}(X)$,  the Hodge symmetry follows.\end{proof}
 Lemma \ref{lm:3.3.1} is   used in the proof of Corollary \ref{cor:3.3.6}. 
  \begin{lm}\label{lm:3.3.1}
For every integer $k\ge0$, the map $H^k(X,\C)\to H^k(X,O_X)$ induced by the inclusion $\C\to O_X$ coincides with the projection $H^k(X,\C)\to H^{0,k}(X)$ given by the Hodge decomposition (\ref{eq:Hodge}).
  \end{lm}
  \begin{proof}
  	Consider the following commutative diagram \begin{center}
  		\begin{tikzcd}
  			\C_X \arrow[r, hook] \arrow[d, hook] & A^0_X \arrow[r, "d"] \arrow[d, "{p^{0,0}}"] & A^1_X \arrow[r, "d"] \arrow[d, "{p^{0,1}}"] & \dots \\
  			O_X \arrow[r, hook]                  & {A^{0,0}_X} \arrow[r, "\db"]                  & {A^{0,1}_X} \arrow[r, "\db"]                  & \dots
  		\end{tikzcd}
  	\end{center}   The first row is an acyclic resolution of $\C_X$ by (smooth) Poincaré lemma, and the second row is the Dolbeault resolution.  The first vertical map is the inclusion and each $p^{0,j}:A^j_X\to A^{0,j}_X$ is taking the $(0,j)$-part of a $j$-form. It is a morphism of complexes.  Taking global sections, the induced map on $k$-th cohomology groups is the first map  in the statement.
  
  For a class $[\alpha]\in H^k(X,\C)$, we may assume that the representative $k$-form $\alpha$  is $\partial$-closed and $\db$-closed by Corollary \ref{cor:decomp}. Then its image under the first map $H^k(X,\C)\to H^k(X,O_X)$ is represented by the $(0,k)$-part of $\alpha$, which is still $\partial$-closed and $\db$-closed. This describes exactly the projection induced by the Hodge decomposition (\ref{eq:Hodge}).
  \end{proof}
\subsection{Jacobian}\label{sec:Pic0}
For  a connected compact complex manifold $X$, let $b_1(X):=\dim_{\C}H^1(X,\C)$ be its first Betti number.  The exponential short exact sequence \[0\to \Z\to O_X\overset{f\mapsto \exp(2\pi if)}{\longrightarrow} O_X^*\to 1\] 
induces a long exact sequence \begin{equation}\label{eq:long}H^0(X,O_X)\overset{f\mapsto \exp(2\pi if)}{\longrightarrow} H^0(X,O_X^*)\to H^1(X,\Z)\to H^1(X,O_X)\to H^1(X,O_X^*)\overset{\delta}{\to}H^2(X,\Z).
\end{equation}
Set $\Pic(X):=H^1(X,O_X^*)$ for the Picard group, $\NS(X):=\im(\delta)$ for the Néron-Severi group, $\Pic^0(X)=\ker(\delta)$ and $\Pic^{\tau}(X):=\delta^{-1}(H^2(X,\Z)_{\tor})$. As $X$ is compact connected, one has $H^0(X,O_X)=\C$, $H^0(X,O_X^*)=\C^*$ and the first map in (\ref{eq:long}) is surjective. Accordingly, the third map $H^1(X,\Z)\to H^1(X,O_X)$ is injective and \begin{equation}\label{eq:Pic0}
	\Pic^0(X)=\frac{H^1(X,O_X)}{H^1(X,\Z)}.
\end{equation}
 If $X$ is a complex torus, then $H^2(X,\Z)$ is torsion free and \begin{equation}\label{eq:Pic0Pictau}
\Pic^0(X)=\Pic^{\tau}(X).
\end{equation}
For general $X$, let $\Loc^1(X)$ (resp. $\Loc^{u,1}(X)$) be the set of isomorphism classes  of  rank-$1$ (resp. and unitary) local systems on $X$. Then $\Loc^1(X)$ is a group under tensor product and $\Loc^{u,1}(X)$ is a subgroup. For  each $\cL\in \Loc^1(X)$,  $L:=\cL\otimes_{\C}O_X$ is a flat line bundle on $X$. By \cite[Ch.~V, \S~9]{demailly1997complex},  $L\in \Pic^{\tau}(X)$, whence a group morphism \begin{equation}\label{eq:LoctoPic}
	\Loc^1(X)\to \Pic^{\tau}(X),\quad \cL\mapsto \cL\otimes_{\C}O_X.
\end{equation} 
\begin{rk}
	Theorem \ref{thm:unitRH} implies that a line bundle on $X$ is 
	flat unitary if and only if its  class in $\Pic(X)$ lies in the image of the restriction of (\ref{eq:LoctoPic}): \begin{equation}\label{eq:LocutoPic}
		\Loc^{u,1}(X)\to \Pic^{\tau}(X).
	\end{equation}
	 The image  of (\ref{eq:LocutoPic}) may not to be contained in $\Pic^0(X)$. For instance, let $X$ be an Enriques surface, then $\pi_1(X,x_0)=\Z/2$, $\#\Loc^{u,1}(X)=\Z/2$. By Corollary \ref{cor:chartoPic} \ref{it:Loc1Pictau} below, the map (\ref{eq:LocutoPic}) is an isomorphism, while $\Pic^0(X)$ is trivial.
\end{rk}

\begin{cor}\label{cor:3.3.6}
Assume  that $X$ is regular. Then	 $\Pic^{\tau}(X)$ has a natural structure of compact complex Lie group with identity component $\Pic^0(X)$  that is a complex torus of  dimension $b_1(X)/2$. Moreover, $\pi_0(\Pic^{\tau}(X))=\NS(X)_{\tor}$.
\end{cor}
\begin{proof}
	The inclusion $\R\subset O_X$ induces an $\R$-linear map \begin{equation}\label{eq:RtoOX}\phi:H^1(X,\R)\to H^1(X,O_X).\end{equation} Because of Lemma \ref{lm:3.3.1} and the Hodge symmetry in Corollary \ref{cor:decomp},  taking complex conjugate inside $H^1(X,\C)$ induces an $\R$-linear map \begin{equation}\label{eq:barphi}\bar{\phi}:H^1(X,\R)\to H^0(X,\Omega_X^1).\end{equation} If $\xi\in \ker(\phi)$, then the image of $\xi$ under the injection $H^1(X,\R)\to H^1(X,\C)$ is $\phi(\xi)+\bar{\phi}(\xi)=0$, so $\xi=0$. This shows that $\phi$ is injective. But $\dim_{\R}H^1(X,\R)=\dim_{\R} H^1(X,O_X)=b_1(X)$, so $\phi$ is a linear isomorphism. 
	
	The map $H^1(X,\Z)\to H^1(X,O_X)$ in (\ref{eq:long}) factors through $\phi$. Since $H^1(X,\Z)$ is a full lattice of $H^1(X,\R)$, it remains a full lattice in $H^1(X,O_X)$. Therefore, the quotient $\Pic^0(X)$ is a complex torus of dimension $b_1(X)/2$. The $\Z$-module $\Pic^0(X)$ is divisible, so the short exact sequence \[0\to \Pic^0(X)\to \Pic^{\tau}(X)\to \NS(X)_{\tor}\to0\] spits. Therefore, there is a natural structure of compact complex Lie group  on $\Pic^{\tau}(X)$ satisfying the stated properties.
\end{proof}
The complex torus $\Pic^0(X)$ in Corollary \ref{cor:3.3.6} is called the Jacobian of the regular manifold $X$.
\begin{eg}\label{eg:notreg}Here are two examples showing how Corollary \ref{cor:3.3.6} fails for non-regular compact complex manifolds.
	\begin{enumerate}
		\item\label{it:Hopf} Let $X$ be a Hopf surface  (\cite[Example 3.3.2]{huybrechts2005complex}). The Betti number $b_1(X)=1$, $H^1(X,\Z)=\Z$ and $H^1(X,O_X)=\C$, so the complex manifold $\Pic^0(X)=\C/\Z$ is not compact. However, by \cite{kodaira1964structure}, the Frölicher spectral sequence of $\C_X$ degenerates.
		\item  Let $Y$ be a Calabi-Eckmann manifold (\cite[Sec 1.2]{biswas2017line}). Then $H^1(Y,O_Y)=\C$ and $Y$ is simply connected, so $H^1(Y,\Z)=0$ and $b_1(Y)=0$, but $\Pic^0(Y)=\C$ is not compact and $b_1(Y)/2<\dim \Pic^0(Y)$.
	\end{enumerate}
\end{eg} 

\section{Albanese torus}\label{sec:Alb}
 We turn to the conception of Albanese torus and Albanese map. They help to reduce some problems about general complex manifolds to those about complex tori. They are also tools to study the Jacobian.   Again, Section \ref{sec:Alb} conveys the fact  that Hodge theory  guarantees the usual properties of the Albanese torus and Albanese map. 

Fix a connected regular manifold $X$ and a base point $x_0\in X$.
\subsection{Basics of Albanese torus}
From \cite[Cor.~9.5, p.101]{ueno2006classification}, every element of $H^0(X,\Omega_X^1)$ is $d$-closed, so there is a well-defined natural map \begin{equation}\label{eq:iota}\iota:H_1(X,\Z)\to H^0(X,\Omega_X^1)^{\vee},\quad [\gamma]\mapsto(\beta\mapsto \int_{\gamma}\beta),\end{equation} where $\gamma$ runs through closed paths on $X$. Set 	\begin{equation}\label{eq:Alb}
	\Alb(X)=H^0(X,\Omega_X^1)^{\vee}/\im(\iota).
\end{equation}
\begin{lm}\label{lm:Albtor}
On	$\Alb(X)$, there is  a natural structure of $h^{1,0}(X)$-dimensional complex torus with $H_1(\Alb(X),\Z)=\im(\iota)$.
\end{lm}
\begin{proof}
Using the $\R$-linear isomorphism (\ref{eq:barphi}) and de Rham isomorphism \[H^1_{\dR}(X,\R)\to H^1(X,\R),\] the map (\ref{eq:iota}) is identified with the natural map $H_1(X,\Z)\to H^1_{\dR}(X,\R)^{\vee}$. The latter extends to an $\R$-linear isomorphism $H_1(X,\R)\to H^1_{\dR}(X,\R)^{\vee}$ by Poincaré duality. Therefore, \begin{equation}\label{eq:keriota}\ker(\iota)=H_1(X,\Z)_{\tor}\end{equation} and $\im(\iota)$ is a full lattice in $H^0(X,\Omega_X^1)^{\vee}$ isomorphic to $H_1(X,\Z)_{\free}$. Thus, the quotient $\Alb(X)$ is a complex torus with the stated properties.
\end{proof}
 The complex torus $\Alb(X)$ in Lemma \ref{lm:Albtor} is called the \emph{Albanese torus} of $X$. For each $x\in X$, choose two paths $\gamma_x$, $\gamma'_x$ connecting $x_0$ to $x$. Then the composition $\gamma$ of $\gamma_x$ followed by the reverse of $\gamma'_x$ is a closed path on $X$ and
\[\int_{\gamma_x}\bullet-\int_{\gamma'_x}\bullet=\int_{\gamma}\bullet=\iota([\gamma])\] belongs to $\im(\iota)$. Therefore, $[\int_{\gamma_x}\bullet]=[\int_{\gamma'_x}\bullet]$ in $\Alb(X)$. As $[\int_{\gamma_x}\bullet]$ is independent of the choice of $\gamma_x$, we write it as $\int_{x_0}^x\bullet$. For the fixed base point $x_0\in X$,  the associated Albanese map is \begin{equation}\label{eq:Albmor}\alpha_{X,x_0}:X\to \Alb(X),\quad   x\mapsto  \int_{x_0}^x\bullet.\end{equation} The subscripts $X$ and $x_0$ are omitted when they are clear from the context.
\begin{pp}\label{pp:Alb}
	\hfill\begin{enumerate}
		\item \label{it:Albcanonical}The Albanese map $\alpha_{X,x_0}:X\to \Alb(X)$ is a morphism of complex manifolds and the formation of Albanese map is functorial for the pair $(X,x_0)$. 
		\item \label{it:AlbH1}The induced morphism $\alpha_{x_0,*}:H_1(X,\Z)\to H_1(\Alb(X),\Z)$ is surjective with kernel $H_1(X,\Z)_{\tor}$.
		\item \label{it:Albuniv}The morphism $\alpha_{x_0}$ satisfies the following universal property:  every morphism of pointed complex manifolds $(X,x_0)\to (A,0)$ with $A$ a complex torus factors uniquely through a morphism of complex tori $\Alb(X)\to A$.
		 In particular, the complex subtorus of $\Alb(X)$ generated by $\alpha_{x_0}(X)$ is $\Alb(X)$.

	\item \label{it:H1=} The pullback morphism $\alpha_{x_0}^*:H^1(\Alb(X),\Z)\to H^1(X,\Z)$ is an isomorphism of weight $1$ $\Z$-Hodge structures independent of the choice of $x_0$.
		\item\label{it:AlbPic0} The pullback $\alpha_{x_0}^*:\Pic^0(\Alb(X))\to \Pic^0(X)$ is an isomorphism of complex tori independent of the choice of $x_0$. In particular, the complex tori $\Alb(X)$ and $\Pic^0(X)$ are  dual to each other.\footnote{in the sense of \cite[p.34]{birkenhake2004complex}}
\end{enumerate}
\end{pp}

\begin{proof}
	\hfill\begin{enumerate}
\item When $X$ is Kähler, it is proved in \cite[Prop.~3.3.8]{huybrechts2005complex}. The general case is similar.
\item By Lemma \ref{lm:Albtor}, $H_1(\Alb(X),\Z)=\im(\iota)$.
Let $\gamma:[0,1]\to X$ be a   closed path on $X$ based at $x_0$.  It defines a path \[\zeta:[0,1]\to H^0(X,\Omega_X^1)^{\vee},\quad \zeta(t)= \int_{\gamma(0)}^{\gamma(t)}\bullet,\] where the  integral is along a part of $\gamma$. Then \[\zeta\pmod{\im(\iota)}=\alpha_{x_0}\circ\gamma:[0,1]\to \Alb(X).\]  Therefore, $\alpha_{x_0,*}[\gamma]=\zeta(1)-\zeta(0)=\int_{\gamma}\bullet=\iota([\gamma])$. Hence a commutative triangle \begin{equation*}
\begin{tikzcd}
	& {\im(\iota)=H_1(\Alb(X),\Z)} \arrow[rd, hook] &                            \\
	{H_1(X,\Z)} \arrow[ru, "{\alpha_{x_0,*}}"] \arrow[rr, "\iota"] &                                              & {H^0(X,\Omega_X^1)^{\vee}} 
                     \end{tikzcd}
\end{equation*}
Therefore, $\alpha_{x_0,*}$ is surjective and $\ker(\alpha_{x_0,*})= \ker(\iota)=H_1(X,\Z)_{\tor}$, where the last equality uses (\ref{eq:keriota}). 
\item The universal property follows from Point \ref{it:Albcanonical}.   Let $T$ be the complex subtorus of $\Alb(X)$ generated by $\alpha_{x_0}(X)$. Then  the pointed morphism $\alpha_{x_0}:(X,x_0)\to (T,0)$ factors through $\alpha_{x_0}:(X,x_0)\to (\Alb(X),0)$, so $T=\Alb(X)$.
\item From \cite[Thm.~1.4.1 b)]{birkenhake2004complex}, the map $\alpha_{x_0}^*:H^{1,0}(\Alb(X))\to H^{1,0}(X)$ is a $\C$-linear isomorphism. By \cite[Sec 1.3, p.13]{birkenhake2004complex}, $H^1(\Alb(X),\Z)$ is naturally isomorphic to $\Hom(\im(\iota),\Z)$. By Poincaré duality, the latter is identified with $H^1(X,\Z)$, so \[\alpha_{x_0}^*:H^1(\Alb(X),\Z)\to H^1(X,\Z)\] is an isomorphism of weight $1$ $\Z$-Hodge structures. Up to translation, different base points give rise to the same Albanese map. More precisely,  for $x\in X$, $T_{\alpha_x(x_0)}\circ \alpha_{x_0}=\alpha_x$, where \[T_a:\Alb(X)\to \Alb(X),\quad u\mapsto u+a\] is the translation by $a$ on $\Alb(X)$. The independence stated in Point \ref{it:H1=} follows.
\item As the isomorphism (\ref{eq:Pic0}) is functorial in $X$, there is a commutative diagram  with exact rows \begin{center}
\begin{tikzcd}
	{H^1(\Alb(X),\Z)} \arrow[r] \arrow[d, "\alpha_{x_0}^*"] & {H^1(\Alb(X),O_{\Alb(X)})} \arrow[r] \arrow[d, "\alpha_{x_0}^*"] & \Pic^0(\Alb(X)) \arrow[r] \arrow[d, "\alpha_{x_0}^*"] & 0 \\
	{H^1(X,\Z)} \arrow[r]                                  & {H^1(X,O_X)} \arrow[r]                                         & \Pic^0(X) \arrow[r]                                  & 0.
\end{tikzcd}
\end{center}
 By   Point \ref{it:H1=}, the left two vertical maps are isomorphisms independent of $x_0$. Therefore, the right vertical map  is an isomorphism independent of $x_0$. As $\Alb(X)$ is a complex torus, by \cite[Proposition 2.4.1]{birkenhake2004complex}, $\Pic^0(\Alb(X))$ is the  dual torus of $\Alb(X)$. As $\alpha_{x_0}^*:\Pic^0(\Alb(X))\to \Pic^0(X)$ is an isomorphism, $\Pic^0(X)$ is dual to $\Alb(X)$.
	\end{enumerate}
\end{proof}
\begin{rk}
By \cite[Cor.~9.5, p.101]{ueno2006classification}, for every connected regular manifold $X$ the formation of $\Alb(X)$ and $\alpha_{x_0}$ agrees with the construction in \cite[\S 2]{blanchard1956varietes}. Then  \cite[p.163]{blanchard1956varietes} gives another proof of the universal property stated in Proposition \ref{pp:Alb}  \ref{it:Albuniv}.
\end{rk}
 \begin{eg_contd}[\ref{eg:notreg} \ref{it:Hopf}]If $X$ were a Hopf surface, then $H_1(X,\Z)=\Z$ and $H^0(X,\Omega_X^1)=0$. Equation (\ref{eq:Alb}) would define a point and Proposition \ref{pp:Alb} \ref{it:AlbH1} would fail.\end{eg_contd}
\subsection{Back to Jacobian}
Albanese torus helps  to understand the Jacobian. Corollary \ref{cor:Poincare}  is used to show the jumping loci are analytic subsets.
\begin{cor}[Universal line bundle]\label{cor:Poincare}There exists a  unique (up to isomorphism) line bundle $L$ on $X\times \Pic^0(X)$ such that its pullback module to $\{x_0\}\times \Pic^0(X)$ is trivial and   
  for every point $y\in \Pic^0(X)$, the isomorphism class of the pullback line bundle $L|_{X\times \{y\}}$ in $\Pic(X)$ is $y$.
\end{cor}
\begin{proof}
Consider the map \[f=\alpha_{x_0}\times \Id_{\Pic^0(X)}:X\times \Pic^0(X)\to \Alb(X)\times \Pic^0(X).\] By Proposition \ref{pp:Alb} \ref{it:AlbPic0}  and \cite[Lemma, p.328]{griffiths2014principles}, there is a holomorphic line bundle $\cP$ on $\Alb(X)\times \Pic^0(X)$ that is trivial on $\{0\}\times \Pic^0(X)$ such that for every $y\in \Pic^0(X)$, the line bundle $\cP|_{\Alb(X)\times \{y\}}$ is of class $y$ in $\Pic^0(\Alb(X))$. Let $L=f^*\cP$, then $L|_{\{x_0\}\times \Pic^0(X)}=f^*(\cP|_{\{0\}\times \Pic^0(X)})$ is trivial. For every $y\in \Pic^0(X)$, the line bundle \[L|_{X\times \{y\}}=f^*(\cP|_{\Alb(X)\times \{y\}})=\alpha_{x_0}^*(\cP|_{\Alb(X)\times \{y\}})\] is of class $y$ in $\Pic^0(X)$. The existence is proved. The uniqueness follows from \cite[Cor.~A.9]{birkenhake2004complex}.\end{proof}
 Let \[\car(X)=\Hom(H_1(X,\Z),\C^*)\]   be the group of characters of the first homology of $X$. By \cite[Cor.~A.8, A.9]{hatcher2005algebraic}, the abelian group $H_1(X,\Z)$ is finitely generated.
 From \cite[Ch.~12 b.]{milne2017algebraic}, $\car(X)$ has a natural structure of diagonalizable algebraic group over $\C$, with identity component $\car^{\circ}(X)$ isomorphic to $\G_m^{b_1(X)}$. 
Moreover, $\car^u(X):=\Hom(H_1(X,\Z), S^1)$  is a real Lie subgroup of $\car(X)$ of  dimension $b_1(X)$.
 There is a canonical group isomorphism by taking  character sheaves  \begin{equation}\label{eq:Piu=Loc}
	\car^u(X)\to \Loc^{u,1}(X),\quad \chi\mapsto \cL_{\chi}.
\end{equation}Set $T(X):=\Hom(H_1(X,\Z)_{\free},S^1)$. Then $T(X)$ is the identity component of $\car^u(X)$. 
From Corollary \ref{cor:3.3.6}, composing the isomorphism (\ref{eq:Piu=Loc}) and the map (\ref{eq:LocutoPic}) gives a  morphism  of real Lie groups
\begin{equation}\label{eq:chartoPic}
T(X)\to \Pic^0(X).
\end{equation}
In   Corollary \ref{cor:chartoPic} \ref{it:TX=Pic0}, the  isomorphism  allows one to identify certain characters with  topologically trivial  line bundles. This identification is used in the proof of Theorem \ref{thm:main}. When $X$ is Kähler, Corollary \ref{cor:chartoPic} \ref{it:Loc1Pictau} is also  in  the proof of  \cite[Cor.~1.4]{wang2016torsion}. 
\begin{cor}\label{cor:chartoPic}
\hfill\begin{enumerate}
\item \label{it:TX=Pic0} The morphism (\ref{eq:chartoPic}) is 
 an isomorphism of real Lie groups.

\item\label{it:Loc1Pictau} The map (\ref{eq:LocutoPic}) is a group isomorphism and $\NS(X)_{\tor}=H^2(X,\Z)_{\tor}$. In particular, every element of $\Pic^{\tau}(X)$ is a flat unitary  line bundle.
\end{enumerate} \end{cor}
\begin{proof} 
\hfill	\begin{enumerate}
\item	Lemma \ref{lm:Albtor} gives an identification $H_1(X,\Z)_{\free}=\im(\iota)$.
By \cite[Prop.~2.2.2]{birkenhake2004complex}, the natural group morphism \begin{equation}\label{eq:AHthm}\Hom(\im(\iota),S^1)\to \Pic^0(\Alb(X))\end{equation} defined \textit{via} factors of automorphy  (\cite[p.30]{birkenhake2004complex}) is an isomorphism. The map (\ref{eq:chartoPic})   is  the composition of (\ref{eq:AHthm}) with the isomorphism $\alpha_{x_0}^*:\Pic^0(\Alb(X))\to \Pic^0(X)$ in Proposition \ref{pp:Alb} \ref{it:AlbPic0}. 

To sum it up:
\begin{center}
\begin{tikzcd}
	{\car^u(\Alb(X))=\Hom(\im(\iota),S^1)} \arrow[r,"\sim"] \arrow[d,"(\ref{eq:AHthm})"] & T(X) \arrow[r, hook] \arrow[d, "(\ref{eq:chartoPic})"] & {\car^u(X)=\Loc^{u,1}(X)} \arrow[d, "(\ref{eq:LocutoPic})"] \\
	\Pic^0(\Alb(X)) \arrow[r, "\alpha_{x_0}^*", "\sim"']                  & \Pic^0(X) \arrow[r, hook]                                                      & \Pic^{\tau}(X).                                                 
\end{tikzcd}
\end{center}
\item The commutative diagram of abelian sheaves on $X$ \begin{center}
\begin{tikzcd}
	0 \arrow[r] & \Z \arrow[r] \arrow[d,"\Id"] & \R \arrow[r] \arrow[d, hook]             & S^1 \arrow[r] \arrow[d, hook] & 0 \\
	0 \arrow[r] & \Z \arrow[r]           & O_X \arrow[r, "f\mapsto \exp(2\pi if)"] & O_X^* \arrow[r]               & 0
\end{tikzcd}
\end{center} has exact rows. Moreover, the $\Z$-module $\R$ is  injective. Therefore, there is a commutative diagram with exact rows \begin{center}
\begin{tikzcd}[sep=0.7em, font=\small]
	&                                                          & 0 \arrow[r]                                       & {\Hom(H_1(X,\Z)_{\tor},S^1)} \arrow[r, "\sim"] \arrow[rdd, "\psi", dotted, bend left] & {\Ext_{\Z}^1(H_1(X,\Z)_{\tor},\Z)} \arrow[r]                            & 0 \\
	0 \arrow[r] & {\Hom(H_1(X,\Z),\Z)} \arrow[r]                           & {\Hom(H_1(X,\Z),\R)} \arrow[r]                    & {\car^u(X)} \arrow[r] \arrow[u, "r"]                                        & {\Ext^1_{\Z}(H_1(X,\Z),\Z)} \arrow[r] \arrow[d, hook,"\xi"] \arrow[u, "\sim"] & 0 \\
	0 \arrow[r] & {H^1(X,\Z)} \arrow[r] \arrow[d, "\Id"] \arrow[u, "\sim"] & {H^1(X,\R)} \arrow[r] \arrow[d, "(\ref{eq:RtoOX})"] \arrow[u, "\sim"] & {H^1(X,S^1)} \arrow[r] \arrow[d] \arrow[u, "\sim"]                            & {H^2(X,\Z)} \arrow[d, "\Id"]                                            &   \\
	0 \arrow[r] & {H^1(X,\Z)} \arrow[r]                                    & {H^1(X,O_X)} \arrow[r]                            & {H^1(X,O_X^*)} \arrow[r, "\delta"]                                                    & {H^2(X,\Z)}                                                             &  
\end{tikzcd}
\end{center} 
where $r$ is the restriction, the  vertical morphisms in the middle are from  \cite[Thm.~3.2]{hatcher2005algebraic} and $\im(\xi)=H^2(X,\Z)_{\tor}$ by \cite[p.196]{hatcher2005algebraic}. Hence an isomorphism $\psi:\Hom(H_1(X,\Z)_{\tor},S^1)\to H^2(X,\Z)_{\tor}$ fitting into a commutative diagram \begin{center}
	\begin{tikzcd}
		0 \arrow[r] & T(X) \arrow[r] \arrow[d,"(\ref{eq:chartoPic})"] & \Loc^{u,1}(X) \arrow[r, "r"] \arrow[d, "(\ref{eq:LocutoPic})"] & {\Hom(H_1(X,\Z)_{\tor},S^1)} \arrow[r] \arrow[d, dashed,"\psi"] & 0  \\
		0 \arrow[r] & \Pic^0(X) \arrow[r]                              & \Pic^{\tau}(X) \arrow[r,"\delta"]                             & H^2(X,\Z)_{\tor} \arrow[r]                                  & 0,
	\end{tikzcd}
\end{center} 
where the first row is exact. Thus, $\delta$ is surjective and the second row is also exact.  By the  five lemma, the middle vertical map (\ref{eq:LocutoPic})  is an isomorphism.
\end{enumerate}
\end{proof}

	\section{Defect of semismallness}\label{sec:defect}
In this section, we review the defect of semismallness of a morphism, an invariant introduced by de Cataldo and Migliorini that plays a crucial role in the decomposition theorem and Lefschetz's theorem. It appears in Fact \ref{ft:KWvanish} and Theorem \ref{thm:main}. Its main property that we need is Proposition \ref{pp:semismall}. 
\subsection{Stratifications and constructible sheaves}
 We refer to \cite[Sec.~2.1]{bingener1984constructible} for  the definitions of \textbf{constructible stratifications} and  \textbf{Whitney stratifications} of a complex analytic space. 
 
 Theorem \ref{thm:semifib} is about the semicontinuity of fiber dimension. Although it is  well-known,  a short proof is included due to the lack of reference.  Its analogue in algebraic geometry is a celebrated theorem of Chevalley \cite[Cor.~13.1.5]{EGAIV3}. 
 \begin{thm}[Analytic Chevalley theorem]\label{thm:semifib}
 	Let	$f:X\to Y$ be a proper morphism of reduced complex analytic spaces.  For every integer $n\ge0$, let $Y_n=\{y\in Y:\dim f^{-1}(y)= n\}$ and $Y_{\ge n}=\cup_{m\ge n}Y_m$. Then
 	$Y_{\ge n}$ is an analytic subset of $Y$. In particular, $\{Y_n\}_{n\in \N}$ is a constructible stratification of $Y$. 
 \end{thm}
 \begin{proof}
 	Let $F_n:=\{x\in X:\dim_xf^{-1}(f(x))\ge n\}$. By \cite[Thm.~3.6, p.137]{fischer2006complex}, $F_n$ is an analytic subset of $X$. By the definition of global dimension \cite[p.94]{grauert2012coherent}, one has $Y_{\ge n}=f(F_n)$. By Remmert theorem (see, \textit{e.g.}, \cite[Thm.~4A, p.150]{whitney1972complex}), the subset $Y_{\ge n}$ is  analytic in $Y$.
 \end{proof}
 
\begin{df}(\cite[p.125]{bingener1984constructible})\label{df:relstr}
Let $f:X\to Y$ be a morphism of complex analytic spaces.  If two Whitney stratifications  $\mathfrak{X}:X=\sqcup_{\alpha}X_{\alpha}$ and $\mathfrak{Y}:Y=\sqcup_{\lambda}Y_{\lambda}$ satisfy that: 
\begin{enumerate}
	\item For each $\alpha$, there is $\lambda$ with $f(X_{\alpha})\subset Y_{\lambda}$;
	\item For each pair $(\alpha,\lambda)$ with $f(X_{\alpha})\subset Y_{\lambda}$, the restricted morphism $f:X_{\alpha}\to Y_{\lambda}$ is smooth.
\end{enumerate}Then  such a pair $(\mathfrak{X},\mathfrak{Y})$ is called a Whitney stratification of $f$.  \end{df}

\begin{ft}[{\cite[Thm.~1]{hironaka1977stratification}, \cite[Lem.~2.4]{bingener1984constructible}, \cite[Thm, p.43]{goresky1988stratified}}]\label{ft:relWhitney}
	Let $f:X\to Y$ be a \emph{proper} morphism of complex analytic spaces. Suppose that $\mathfrak{X}$, (resp. $\mathfrak{Y}$) is a constructible stratification of $X$ (resp. $Y$), then there exists a Whitney stratification $(\mathfrak{X}',\mathfrak{Y}')$ of $f$ such that $\mathfrak{X}'$ (resp. $\mathfrak{Y}'$) refines $\mathfrak{X}$ (resp. $\mathfrak{Y}$).
\end{ft}
Corollary \ref{cor:common} is useful but implicit in the literature.
\begin{cor}\label{cor:common}
Let $X$ be a complex analytic space. For finitely many constructible stratifications of $X$, there exists a Whitney stratification of $X$ refining all of them.
\end{cor}
\begin{proof}
It suffices to  consider the case of two constructible stratifications $\mathfrak{X}_1$ and $\mathfrak{X}_2$ of $X$. By Fact \ref{ft:relWhitney}, there is a  Whitney stratification $(\mathfrak{X},\mathfrak{X}')$ of $\Id_X$ such that $\mathfrak{X}$ (resp. $\mathfrak{X}'$) refines $\mathfrak{X}_1$ (resp. $\mathfrak{X}_2$). Moreover, $\mathfrak{X}$ refines $\mathfrak{X}'$ by Definition \ref{df:relstr}.  Hence a Whitney stratification $\mathfrak{X}$ refining both $\mathfrak{X}_1$ and $\mathfrak{X}_2$.
\end{proof}
For a complex analytic space $X$, using analytic constructible stratifications, one can define constructible sheaves.  Let $D_c^b(X)$ be the triangulated category of complexes of  sheaves of $\C$-vector spaces whose  cohomology is bounded and constructible (see, \textit{e.g.}, \cite[p.82]{dimca2004sheaves}).  
\begin{ft}[{\cite[Prop.~8.5.7  (b)]{kashiwara2013sheaves}, \cite[Thm.~4.1.5 (b)]{dimca2004sheaves}}]\label{ft:amplitude}
	Let $f:X\to Y$ be a   morphism of complex analytic spaces and $\cK\in D_c^b(X)$. If  $f$ is proper on $\mathrm{Supp}(\cK)$, then $Rf_*\cK\in D_c^b(Y)$.
\end{ft}
\begin{cor}\label{cor:allloc}
Let $f:X\to Y$ be a \emph{proper} morphism of complex analytic spaces and $\cK\in D_c^b(X)$. Then there exists a Whitney stratification $(\mathfrak{X},\mathfrak{Y})$ of $f$ such that for every integer $i$ and every stratum $S$ of $\mathfrak{Y}$, the restriction $\cH^i(Rf_*\cK)|_S$ is a local system on $S$. 
\end{cor}
\begin{proof}
By Fact \ref{ft:amplitude}, $Rf_*\cK\in D_c^b(Y)$. In particular, there are only finitely many $j\in \Z$ with $\cH^j(Rf_*\cK)\neq0$. For  each such $j$, there is an admissible partition (in the sense of \cite[p.81]{dimca2004sheaves}) $\cP_j$ on $Y$ such that the restriction of $\cH^j(Rf_*\cK)$ to each stratum of $\cP_j$ is a local system. By Corollary \ref{cor:common}, there exists a Whitney stratification  $\mathfrak{Y}^0$ of $Y$ refining the finitely many $\cP_j$. By Fact \ref{ft:relWhitney}, there is  a Whitney stratification $(\mathfrak{X},\mathfrak{Y})$ of $f$ satisfying the properties.
\end{proof}
\subsection{Equivalent definitions}\label{sec:equivdf}The defect of semismallness  measures how far a morphism of complex manifolds is from being semismall (see, \textit{e.g.}, \cite[Def.~7.3, p.156]{kiehl2001weil}).  However, in the literature there exist multiple seemingly different definitions. We review some of them and show that they are equivalent.
\begin{df}\label{df:defect}
	Let	$f:X\to Y$ be a proper morphism of complex manifolds with $\dim X=n$. \begin{itemize}
		\item(\cite[Definition 1.1]{esnault1987vanishing}) Define \begin{equation}\label{eq:r1}r_1(f)=\max_Z(\dim Z-\dim f(Z)-\codim_X(Z)),
		\end{equation} where $Z$ runs through all  irreducible analytic subsets of $X$.
		\item(\cite[Definition 9.3.7]{maxim2019intersection}) For a Whitney stratification $(X=\sqcup S_{\alpha},Y=\sqcup T_{\lambda})$ of $f$, we choose a point $y_{\lambda}\in T_{\lambda}$ in each stratum, and define \begin{equation}\label{eq:defectlam}
			r_2(f)=\max_{\lambda}\{2\dim f^{-1}(y_{\lambda})+\dim T_{\lambda}-n\}.
		\end{equation} (By convention, the empty space has dimension $-\infty$.)
		\item(\cite[Definition 4.7.2]{de2005hodge}) For each integer $i\ge0$, let $Y_i=\{y\in Y:\dim f^{-1}(y)=i\}$.  Define \[r_3(f)=\max_{i\ge0}(2i+\dim Y_i-n).\] 
		\item(\cite[Definition 2.8]{popa2013generic}) For each integer $i\ge0$, let  $Y_{\ge i}=\{y\in Y:\dim f^{-1}(y)\ge i\}$ for each $i\ge0$.  Define \[r_4(f)=\max_{i\ge0}(2i+\dim Y_{\ge i}-n).\]
		\item(\cite[Sec. 3.3.2, part 2]{de2009decomposition}) Define \begin{equation}\label{eq:defectdim}
			r_5(f)=\dim X\times_YX-n.
		\end{equation}
	\item(\cite[Sec 3.2]{williamson2016hodge}) Define \[r_6(f)=\max\{i\in \Z:{}^p\cH^i(Rf_*\C_X[n])\neq0\}.\]
	\end{itemize}
\end{df}
\begin{pp} The first five numbers in Definition \ref{df:defect} are all equal. 
\end{pp}
This common integer is called the \emph{defect of semismallness} of $f$ and denoted by $r(f)$. We  shall show $r(f)=r_6(f)$ in Proposition \ref{pp:semismall} \ref{it:r6}.
\begin{proof}
\hfill\begin{itemize}
\item $r_3(f)=r_4(f)$: As each $Y_i$ is a subset of $Y_{\ge i}$, one has $r_3(f)\le r_4(f)$. There are only finitely many integers $i\ge0$ with $Y_{\ge i}$ nonempty, so the maximum defining $r_4(f)$ is attained at some $i_0(\ge0)$. Then \[2(i_0+1)+\dim Y_{\ge i_0+1}\le 2i_0+\dim Y_{\ge i_0}.\] Since $Y_{\ge i_0}=Y_{\ge i_0+1}\cup Y_{i_0}$, one has $\dim Y_{\ge i_0}=\dim Y_{i_0}$. Then \[r_4(f)=2i_0+\dim Y_{\ge i_0}-n\le r_3(f).\] Therefore, $r_3(f)=r_4(f)$.
\item $r_2(f)=r_5(f)$: By Thom's first isotopy lemma (see, \textit{e.g.}, \cite[Prop.~11.1]{Mather2012Notes}), for every  $\lambda$, the restriction $f|_{f^{-1}(T_{\lambda})}:f^{-1}(T_{\lambda})\to T_{\lambda}$ is a topologically locally trivial fibration. Therefore, $\dim f^{-1}(y_{\lambda})$ is independent of $y_{\lambda}\in T_{\lambda}$ and \begin{equation}\label{eq:Tlam}
		\dim f^{-1}(T_{\lambda})\times_{T_{\lambda}}f^{-1}(T_{\lambda})=\dim T_{\lambda}+2\dim f^{-1}(y_{\lambda}).
	\end{equation}
	As $\{f^{-1}(T_{\lambda})\times_{T_{\lambda}}f^{-1}(T_{\lambda})\}_{\lambda}$ is a locally finite  partition of $X\times_YX$ into 
	locally closed subsets (in the analytic Zariski topology), one has \begin{equation}\label{eq:partition}\dim X\times_YX=\max_{\lambda}[\dim f^{-1}(T_{\lambda})\times_{T_{\lambda}}f^{-1}(T_{\lambda})].\end{equation}
	Plugging (\ref{eq:Tlam})
	into (\ref{eq:partition}) we get $r_5(f)=r_2(f)$. In particular, $r_2(f)$ is independent of the choice of the stratifications.
	
	\item $r_1(f)\le r_2(f)$: For every irreducible analytic subset $Z\subset X$, $f(Z)$ is an irreducible analytic subset of $Y$. Then $\{Y\setminus f(Z),f(Z)\}$ is a constructible stratification of $Y$.  Fact \ref{ft:relWhitney} yields a Whitney stratification $(X=\sqcup S_{\alpha},Y=\sqcup T_{\lambda})$ of $f$ with $Y=\sqcup T_{\lambda}$ refining $\{Y\setminus f(Z),f(Z)\}$. There exists $\lambda_0$ such that $T_{\lambda_0}$ is an open subset of $f(Z)$, hence $\dim T_{\lambda_0}\le \dim f(Z)$. Then $f^{-1}(T_{\lambda_0})\cap Z$ is a nonempty open subset of $Z$. Therefore, \[\dim Z=\dim(f^{-1}(T_{\lambda_0})\cap Z)\le \dim f^{-1}(T_{\lambda_0}).\] Then \[2\dim Z-\dim f(Z)\le 2\dim f^{-1}(T_{\lambda_0})-\dim T_{\lambda_0}=2\dim f^{-1}(y_{\lambda_0})+\dim T_{\lambda_0}.\] This shows $r_1(f)\le r_2(f)$. In particular, the maximum in (\ref{eq:r1}) is indeed attained.
	
\item $r_2(f)\le r_1(f)$: Fix a Whitney stratification $Y=\sqcup_{\lambda}T_{\lambda}$  defining $r_2(f)$. For every $\lambda$ with $f^{-1}(y_{\lambda})$ nonempty, $\overline{T_{\lambda}}$ is an analytic subset of $Y$ of dimension $\dim T_{\lambda}$. Then $f^{-1}(\overline{T_{\lambda}})$ is a nonempty analytic subset of $X$. Let $Z_0$ be an irreducible component of $f^{-1}(\overline{T_{\lambda}})$  with $\dim Z_0=\dim f^{-1}(\overline{T_{\lambda}})$. Then $f(Z_0)\subset \overline{T_{\lambda}}$ and $\dim f(Z_0)\le \dim T_{\lambda}$. Therefore, \[2\dim f^{-1}(y_{\lambda})+\dim T_{\lambda}=2\dim f^{-1}(T_{\lambda})-\dim T_{\lambda}\le 2\dim Z_0-\dim f(Z_0).\] This shows $r_2(f)\le r_1(f)$.
 \item  $r_2(f)\le r_3(f)$: By  Theorem \ref{thm:semifib}, $\{Y_i\}$ is a constructible stratification of $Y$. By Fact \ref{ft:relWhitney}, there is a   Whitney stratification $(X=\sqcup S_{\alpha},Y=\sqcup T_{\lambda})$ of $f$ such that the stratification $Y=\sqcup T_{\lambda}$ refines $Y=\sqcup_iY_i$. For every $\lambda$, there is $i_0$ with $T_{\lambda}\subset Y_{i_0}$. In particular, for every $y_{\lambda}\in T_{\lambda}$, one has $\dim f^{-1}(y_{\lambda})=i_0$, so \[2\dim f^{-1}(y_{\lambda})+\dim T_{\lambda}\le 2i_0+\dim Y_{i_0}.\] This shows   $r_2(f)\le r_3(f)$. 
	
	\item $r_3(f)\le r_2(f)$: For every integer $i\ge0$ with $Y_i$ nonempty, $Y_i=\sqcup_{\lambda}(Y_i\cap T_{\lambda})$ is a constructible stratification, so there is an index $\lambda_0$ with $\dim (Y_i\cap T_{\lambda_0})=\dim Y_i$. Then $\dim Y_i\le \dim T_{\lambda_0}$. One may take $y_{\lambda_0}\in Y_i\cap T_{\lambda_0}$.
	Then \[2i+\dim Y_i\le 2\dim f^{-1}(y_{\lambda_0})+\dim T_{\lambda_0},\] which shows $r_3(f)\le r_2(f)$. \end{itemize}
\end{proof}
From the diagonal inclusion $X\to X\times_YX$, one gets $\dim X\le \dim X\times_YX$, so $r(f)=r_5(f)\ge 0$.  If $r(f)=0$, then $f$ is said to be semismall.
\begin{eg}
\hfill\begin{enumerate}
\item If $f:X\to Y$ is a proper morphism of complex manifolds that is flat of relative dimension $r$, then $r(f)=r$. 
\item Let $X$ be projective manifold such that $-K_X$ is nef and $\alpha:X\to \Alb(X)$ be the Albanese map associated with some base point. Then $r(\alpha)=\dim X-\dim \alpha(X)$ by \cite[Theorem]{lu2010semistability}. 
\end{enumerate}
\end{eg}

\subsection{Direct image of local systems}
  Defect of semismallness is an important invariant appearing in the decomposition of direct image of perverse sheaves. Proposition \ref{pp:semismall} is an elementary instance. We begin with  a well-known estimation of  cohomological dimension of a complex analytic space, used in the proof of Proposition \ref{pp:semismall}. An analogue for topological manifolds is  \cite[Prop.~3.2.2 (iv)]{kashiwara2013sheaves}.
\begin{lm}\label{lm:cdofana}
	Let $X$ be a paracompact\footnote{in the sense of \cite[Def., p.253]{munkres2000topology}} complex analytic space of complex dimension $n$. Then $H^q(X,F)=0$ for every abelian sheaf $F$ on $X$ and every integer $q>2n$.
\end{lm}
\begin{proof}
	By \cite[Prop., p.94]{grauert2012coherent}, there is an open covering $\{U_{\alpha}\}_{\alpha}$ of $X$ such that for each $\alpha$, there is a finite morphism $f_{\alpha}:U_{\alpha}\to B_{\alpha}$ of complex analytic spaces to an open ball  $B_{\alpha}\subset \C^n$. As $X$ is Hausdorff paracompact, by  \cite[Lemma 41.6]{munkres2000topology}, there exists a locally finite open covering $\{V_{\alpha}\}$ on $X$ such that $\overline{V_{\alpha}}\subset U_{\alpha}$ for each $\alpha$. 
	
	From  \cite[p.314]{munkres2000topology}, the topological dimension (\cite[Def., p.305]{munkres2000topology})  $\covdim (B_{\alpha})=2n$ for every $\alpha$. By \cite[Prop.~51 A.2]{kaup2011holomorphic}, the topological space $X$ is metrizable. From \cite[Thm.~32.2]{munkres2000topology}, each $U_{\alpha}$ is normal.
	 Therefore, by  \cite[Thm.~3.3.10, p.200]{engelking1995theory}, $\covdim (U_{\alpha})\le 2n$. By \cite[Theorem 3.1.3, p.169]{engelking1995theory}, $\covdim (\overline{V_{\alpha}})\le 2n$. Similarly, $X$ is normal, so  $\covdim (X)\le 2n$ by \cite[Thm.~3.1.10, p.172]{engelking1995theory}. By Alexandroff theorem (see, \textit{e.g.}, \cite[p.122]{bredon2012sheaf}), the cohomological dimension (\cite[p.75]{engelking1995theory}) $\dim_{\Z}X\le 2n$. 
\end{proof}
The category $D_c^b(X)$ has a natural perverse t-structure ($p$ being the middle perversity) \[({}^pD^{\le0}(X),{}^pD^{\ge0}(X)),\] whose heart $\Perv(X)$ is a $\C$-linear abelian category (\cite{beilinson2018faisceaux}, see also \cite[Thm.~8.1.27]{hotta2007d}). An object of $\Perv(X)$ is called a perverse sheaf on $X$.
For every  integer $i$, the functor taking the $i$-th perverse cohomology sheaf is denoted by ${}^p\cH^i:D_c^b(X)\to \Perv(X)$.
For any two integers $a\le b$,  set \begin{gather*}{}^pD^{[a,b]}(X):=\{\cK\in D_c^b(X):{}^p\cH^i(\cK)=0,\forall i\notin[a,b]\};\\
	D^{[a,b]}(X):=\{\cK\in D_c^b(X):\cH^i(\cK)=0,\forall i\notin[a,b]\}.
\end{gather*}

Verdier duality $\cD_X:D_c^b(X)\to D_c^b(X)$ is a contravariant auto-equivalence that interchanges ${}^pD^{\le0}(X)$ and ${}^pD^{\ge0}(X)$ (see, \textit{e.g.}, \cite[p.192]{hotta2007d}).

Proposition \ref{pp:semismall}  is an analytic analogue of \cite[Prop.~10.0.7]{de2003alg}. It  allows local coefficients and in our case permits  to descend  some problems  about  local systems on $X$ to  problems about  complexes of sheaves on $Y$. The proof is different from that in \cite{de2003alg}, in particular it does not use the  decomposition theorem \cite[Thm.~10.0.6]{de2003alg}.
\begin{pp}\label{pp:semismall}
	Let $f:X\to Y$ be a \emph{proper}  morphism of complex manifolds, where $X$ is of pure dimension $n$. Let $\cL$ a local system  on $X$. Then:
	\begin{enumerate}
		\item\label{it:withinrf}   $Rf_*(\cL[n])\in {}^pD^{[-r(f),r(f)]}(Y)$. In particular, $Rf_*\cL[n]\in\Perv(Y)$ when $f$ is moreover semismall. 
		\item\label{it:r6} When $\cL=\C_X$, for $j=\pm r(f)$, one has ${}^p\cH^j(Rf_*\C_X[n])\neq0$. In particular,  \begin{equation}\label{eq:r=r6}
			r(f)=r_6(f).
		\end{equation}
	\end{enumerate}
\end{pp}

\begin{proof}
From Corollary \ref{cor:allloc}, there exists a Whitney stratifications $(X=\sqcup_{\alpha}X_{\alpha},Y=\sqcup_{\lambda}Y_{\lambda})$ of $f$ such that  for every $\lambda$, every integer $j$, the  restriction $\cH^j(Rf_*\cL[n])|_{Y_{\lambda}}$ is a local system. For each $\lambda$, choose a point $y_{\lambda}\in Y_{\lambda}$.
	\begin{enumerate}
		\item First, we show that $Rf_*\cL[n]\in {}^pD^{\le r(f)}(Y)$. Fix  an integer $i$. 	If $\dim Y_{\lambda}>r(f)-i$, then by (\ref{eq:defectlam}), one has $i+n>2\dim f^{-1}(y_{\lambda})$. Since the fiber $f^{-1}(y_{\lambda})$ is a compact complex analytic space,  by Lemma \ref{lm:cdofana}, \[H^{i+n}(f^{-1}(y_{\lambda}),\cL|_{f^{-1}(y_{\lambda})})=0.\] 
		By proper base change theorem (see, \textit{e.g.}, \cite[Thm.~17.2]{milneLEC}), \[\cH^i(Rf_*\cL[n])_{y_{\lambda}}=H^{i+n}(f^{-1}(y_{\lambda}),\cL|_{f^{-1}(y_{\lambda})}).\] So $\cH^i(Rf_*\cL[n])=0$ on every stratum $Y_{\lambda}$ with $\dim Y_{\lambda}>r(f)-i$. Therefore, $\dim \Supp \cH^i(Rf_*\cL[n])\le r(f)-i$ and hence $Rf_*\cL[n]\in {}^pD^{\le r(f)}(Y)$. 
		
		It remains to show $Rf_*\cL[n]\in {}^pD^{\ge -r(f)}(Y)$. By what we have proved, $Rf_*\cL^{\vee}[n]\in {}^pD^{\le r(f)}(Y)$. Since $\cD_X(\cL[n])=\cL^{\vee}[n]$, 
		one has \[Rf_*\cL^{\vee}[n]=Rf_*\cD_X(\cL[n])=\cD_Y(Rf_*\cL[n]).\] The last  equality uses Verdier's duality (see, \textit{e.g.}, \cite[Prop.~5.3.9]{maxim2019intersection}). This shows  $Rf_*\cL[n]\in {}^pD^{\ge-r(f)}(Y)$.
		
	\item By (\ref{eq:defectlam}), there exists $\lambda_0$ with $r(f)=2\dim f^{-1}(y_{\lambda_0})+\dim Y_{\lambda_0}-n$. In particular, $f^{-1}(y_{\lambda_0})$ is nonempty. Let $i_0=r(f)-\dim Y_{\lambda_0}$, then $i_0+n=2\dim f^{-1}(y_{\lambda_0})$. By proper base change theorem again, \[\cH^{i_0}(Rf_*\C[n])_{y_{\lambda}}=H^{i_0+n}(f^{-1}(y_{\lambda}),\C)\neq0.\] Therefore, $Y_{\lambda_0}\subset \Supp \cH^{i_0}(Rf_*\C[n])$ and hence \[\dim \Supp \cH^{i_0}(Rf_*\C[n])\ge \dim Y_{\lambda_0}=r(f)-i_0.\] Then $Rf_*\C[n]\notin {}^pD^{\le r(f)-1}(Y)$. Together with Point \ref{it:withinrf}, this shows \[{}^p\cH^{r(f)}(Rf_*\C_X[n])\neq0.\] The other part follows from Verdier's duality.
\end{enumerate}\end{proof}

	\section{Generic vanishing for constructible sheaves}\label{sec:KW}
In Section \ref{sec:KW}, we review the generic vanishing theorem for (complexes of) constructible sheaves on a complex torus. The  case of abelian varieties is treated in \cite{kramer2015vanish} and the general case in \cite{bhatt2018vanishing}. We shall reduce the generic vanishing problem on a manifold in Fujiki class $\cC$  to  results on its Albanese torus.
\subsection{Thin subsets}
To state Krämer-Weissauer's theorem, we recall the terminology ``thin subset" introduced in \cite[p.532 and p.536]{kramer2015vanish}.

Fix a complex torus $A$. Then $\car(A)$ has a natural structure of algebraic torus over $\C$ of dimension $2\dim A$ and $T(A)=\car^u(A)$ is a real Lie subgroup of $\car(A)$.  For each complex subtorus $B\subset A$, let $K(B)$ be the kernel of the morphism of algebraic tori $\car(A)\to \car(B)$ induced by functoriality.  The induced morphism $\pi_1(B,0)\to \pi_1(A,0)$ is  injective with torsion-free cokernel  of rank $2\dim A-2\dim B$, so $K(B)$ is an  algebraic subtorus of $\car(A)$.
\begin{df}A thin subset of $\car(A)$ is a finite union of translates $\chi_i\cdot K(A_i)$ for certain characters $\chi_i\in \car(A)$ and certain \emph{nonzero} complex subtori $A_i\subset A$.  If every $\chi_i$ can be chosen to be a torsion point of $\car(A)$, then such a thin subset is called \emph{arithmetic}.\end{df}
A thin subset of $\car(A)$ is  strict and Zariski closed. If the complex torus $A$ is nonzero and \emph{simple}, then a subset of $\car(A)$ is thin if and only if it is finite.

For each complex subtorus $B\subset A$,  we have a functorial commutative diagram \begin{equation}\label{cd:CLP}
\begin{tikzcd}
	\car^u(A) \arrow[r, "(\ref{eq:Piu=Loc})"] \arrow[d,"\phi"] & {\Loc^{u,1}(A)} \arrow[r, "(\ref{eq:LocutoPic})"] \arrow[d] & \Pic^{\tau}(A) \arrow[d] & \Pic^0(A) \arrow[d,"\psi"] \arrow[l, "(\ref{eq:Pic0Pictau})"] \\
	\car^u(B) \arrow[r]                                 & {\Loc^{u,1}(B)} \arrow[r]                                  & \Pic^{\tau}(A)           & \Pic^0(B) \arrow[l]                                   
\end{tikzcd}
\end{equation} where all the horizontal maps are isomorphisms by Corollary \ref{cor:chartoPic} \ref{it:Loc1Pictau}.

A subset of $\Pic^0(A)$ is called (arithmetic and) thin, if it is the intersection of $\car^u(A)$ with a (arithmetic and) thin subset of $\car(A)$ when $\Pic^0(A)$ is identified with $\car^u(A)$ \textit{via} the diagram (\ref{cd:CLP}). 

\begin{lm}\label{lm:thintrans}Every thin subset of $\Pic^0(A)$ is a finite union of translates of strict complex subtori. \end{lm}\begin{proof}Let $B$ be a subtorus of $A$. As the induced morphism $\pi_1(B,0)\to \pi_1(A,0)$ is  injective with torsion-free cokernel  of rank $2(\dim A-\dim B)$,  the  restriction morphism $\phi:\car^u(A)\to \car^u(B)$ in   (\ref{cd:CLP}) is  surjective, and its kernel  $K(B)\cap \car^u(A)$ is the group of unitary characters of $\pi_1(A,0)/\pi_1(B,0)$. Therefore,
the kernel of the   morphism $\psi:\Pic^0(A)\to \Pic^0(B)$ in  (\ref{cd:CLP}) is a complex subtorus of dimension $\dim A-\dim B$. \end{proof}

For a connected regular manifold $X$, let $\alpha:X\to \Alb(X)$ be its Albanese morphism corresponding to some base point. Then $\alpha$ induces a morphism $\alpha^*:\car(\Alb(X))\to \car(X)$ of algebraic groups. By Proposition \ref{pp:Alb} \ref{it:AlbH1}, this map identifies $\car(\Alb(X))$ with the identity component $\car^{\circ}(X)$ of $\car(X)$. Thus we can define thin subsets of $\car^{\circ}(X)$. By Proposition \ref{pp:Alb} \ref{it:AlbPic0}, $\Pic^0(X)$ is naturally identified with $\Pic^0(\Alb(X))$, thus we can define (arithmetic and) thin subsets of $\Pic^0(X)$.
\subsection{Generic vanishing result on regular manifolds}
Roughly speaking, Krämer-Weissauer's theorem controls the failure of  vanishing for perverse sheaves on complex tori, measured by the following loci.

Let $X$ be a  compact complex manifold of dimension $d$. For any integers $k\ge0$, $i$ and for every  $\cK\in D_c^b(X)$,  consider the cohomology support locus  \[\Sigma^i(X,\cK):=\{\chi\in \car(X): H^i(X,\cL_{\chi}\otimes \cK)\neq0\}.\] 
Let $\Sigma^{\neq0}(X,K):=\cup_{i\neq0,i\in \Z}\Sigma^i(X,K)$. Similarly, let $\Sigma^{>j}(X,K):=\cup_{i>j}\Sigma^i(X,K)$ for every integer $j$.
By Verdier's duality, $H^{2d-i}(X,\cK^{\vee}\otimes \cL_{\chi^{-1}})$ is the $\C$-linear dual of $H^i(X,\cK\otimes \cL_{\chi})$. Therefore,
\begin{equation}\label{eq:Sigmadual}
	\Sigma^{2d-i}(X,\cK^{\vee})=\{\chi^{-1}:\chi\in \Sigma^i(X,\cK)\}.
\end{equation}

\begin{ft}\label{ft:vanish}Let $X$ be a compact Kähler manifold, $\cK\in D_c^b(X)$. Then: \begin{enumerate}
		\item \textup{({\cite[p.547]{wang2016torsion}})} \label{it:cls} For every integer $i$, the subset $\Sigma^i(X,\cK)$ of $\car(X)$ is Zariski closed.
		\item \textup{(\cite[Thm.~1.1]{bhatt2018vanishing})}\label{it:torus} If $X$ is a complex torus, and if $\cK\in \Perv(X)$, then $\Sigma^{\neq0}(X,\cK)$ is a strict subset of $\car(X)$.
		\item \textup{(\cite[Thm.~1.1 and Lem.~11.2 (c)]{kramer2015vanish})}\label{it:abel} If  $X$ is a complex abelian variety, then $\Sigma^{\neq0}(X,\cK)$ is contained in a thin  (and  arithmetic when $\cK$ is semisimple of geometric origin\cref{foot:geoori}) subset of $\car(X)$. \end{enumerate}
\end{ft}
\begin{cor}\label{cor:speccls}Let $X$ be a compact Kähler manifold, and let $\cK\in D_c^b(X)$.  Then: \begin{enumerate}
		\item	 There are only finitely many integers $i$ such that $\Sigma^i(X,\cK)\neq\emptyset$. In particular,
		$\Sigma^{\neq0}(X,\cK)$ and for every integer $j$, $\Sigma^{>j}(X,\cK)$ are Zariski closed in $\car(X)$.
		\item \label{it:Sigma>m}If $X$ is  a complex torus,  and if $\cK\in {}^pD^{\le m}(X)$ for  some integer $m$, then $\Sigma^{>m}(X,\cK)\neq\car(X)$.
		\item \label{it:Sigmathin} If $X$ is a complex abelian variety,  and $\cK\in {}^pD^{\le m}(X)$ for  some integer $m$, then $\Sigma^{>m}(X,\cK)$ is contained in a thin (and arithmetic when $\cK$ is semisimple of geometric origin) subset of $\car(X)$.\end{enumerate}
\end{cor}
\begin{proof}The proof is sketched in \cite[p.533]{kramer2015vanish}. 
	\begin{enumerate}
		\item There exist two integers $c<d$ such that $\cK\in D^{[c,d]}(X)$. Applying \cite[Proposition 10.2.12]{kashiwara2013sheaves} to the proper morphism $X\to p$, where $p$ is a point, one gets two integers $a<b$ such that $Rf_*(D^{[c,d]}(X))\subset D^{[a,b]}(p)$. For every  character sheaf $\cL$  on $X$, the functor $*\otimes^L \cL:D_c^b(X)\to D_c^b(X)$ is  t-exact with respect to the standard t-structure. Consequently,  $\cK\otimes^L \cL\in D^{[c,d]}(X)$ and hence $Rf_*(\cK\otimes^L \cL)\in D^{[a,b]}(p)$. For all integers $i\notin [a,b]$, $\Sigma^i(X,\cK)=\emptyset$. This shows the first part of the assertion. The second part of the assertion follows from Fact \ref{ft:vanish} \ref{it:cls}.

		\item By shifting degree, one may assume  $m=0$. For every  character sheaf $\cL$  on $X$,  the functor $*\otimes^L \cL:D_c^b(X)\to D_c^b(X)$ is t-exact with respect to the perverse t-structure (\cite[Prop.~4.1]{kramer2015vanish}). Hence for every integer $j$, ${}^p\cH^j(\cK\otimes^L \cL)={}^p\cH^j(\cK)\otimes^L \cL$.
	Consider the subset \begin{equation}\label{eq:W}W=\cup_{j\in \Z}\Sigma^{\neq0}(X,{}^p\cH^j(\cK))\end{equation}of $\car(X)$. It is in fact a finite union, because by \cite[Remark 5.1.19]{dimca2004sheaves}, ${}^p\cH^j(K)\neq0$ for only finitely many integers $j$. By Fact \ref{ft:vanish} \ref{it:torus}, $W\neq \car(X)$. 
	
	For  every $\chi\in \car(X)\setminus W$, consider the Grothendieck spectral sequence from \cite[p.545]{de2009decomposition} \begin{equation}\label{eq:Gro}E_2^{i,j}=H^i(X,{}^p\cH^j(\cK)\otimes^L  \cL_{\chi})\Rightarrow H^{i+j}(X,\cK\otimes^L  \cL_{\chi}).
	\end{equation}
	
	For any integers $i\neq0$ and $j$, one has $H^i(X,{}^p\cH^j(\cK)\otimes^L \cL_{\chi})=0$, so the spectral sequence (\ref{eq:Gro})
	degenerates\footnote{in the sense of \cite[\href{https://stacks.math.columbia.edu/tag/011O}{Tag 011O (2)}]{stacks-project}} at page $E_2$ and hence \[H^j(X,\cK\otimes^L \cL_{\chi})=H^0(X,{}^p\cH^j(\cK)\otimes^L \cL_{\chi})\] for every integer $j$. Now that $
	\cK\in {}^pD^{\le0}(X)$, for every $i>0$ one has  ${}^p\cH^i(\cK)=0$ and hence $H^i(X,\cK\otimes^L \cL_{\chi})=0$. This shows $\chi\notin \Sigma^{>0}(X,\cK)$. One concludes that $\Sigma^{>0}(X,\cK)\subset W$.  
	\item As ${}^p\cH^j(\cK)\neq0$ for only finitely many integers $j$, by Fact \ref{ft:vanish} \ref{it:abel},  the subset $W$ defined by (\ref{eq:W}) is contained in a thin (and arithmetic when $\cK$ is semisimple of geometric origin) subset of $\car(X)$.
	\end{enumerate}
\end{proof}

Theorem \ref{thm:premain} is a   generic vanishing result for local systems on a manifold admitting Hodge theory. When  $X$ is a projective manifold,  \cite[Theorem 1.5]{popa2013generic} gives  a dimension estimate of $\Sigma^k(X,\C_X)$.
\begin{thm}\label{thm:premain}
	Let $X$ be a connected regular manifold of  dimension $n$. Let $\alpha:X\to \Alb(X)$ be the Albanese map associated with some base point and $\cE$ be a  local system on $X$. Let    $k$ be an integer either $<n-r(\alpha)$ or $>n+r(\alpha)$. Then: \begin{enumerate}
		\item\label{it:Sigstri} $\Sigma^k(X,\cE)\cap \car^{\circ}(X)$ is a strict Zariski closed subset of $\car^{\circ}(X)$.
		\item \label{it:thin} If furthermore $\Alb(X)$ is algebraic, then $\Sigma^k(X,\cE)\cap \car^{\circ}(X)$ is
		contained in a  thin subset of $\car^{\circ}(X)$.
	\end{enumerate}
\end{thm}
\begin{proof}In view of (\ref{eq:Sigmadual}), one may assume  $k>d+r(\alpha)$. Set $\cK:=R\alpha_*\cE[d+r(\alpha)]$. We first prove \begin{equation}\label{eq:SigmaXAlb}\Sigma^k(X,\cE)\cap \car^{\circ}(X)=\Sigma^{k-d-r(\alpha)}(\Alb(X),\cK)\subset \Sigma^{>0}(\Alb(X),\cK).\end{equation}
	This is used in the proof of both \ref{it:Sigstri} and \ref{it:thin}.
	
	Indeed, by Proposition \ref{pp:semismall}, the complex of sheaves $\cK$ lies in ${}^pD^{\le0}(\Alb(X))$. For every $\chi\in \car^{\circ}(X)$, let $\cD_{\chi}$ (resp. $\cL_{\chi}$)  be the corresponding  character sheaf on $\Alb(X)$ (resp. on $X$). Then $\alpha^*\cD_{\chi}=\cL_{\chi}$. By \cite[Cor.~7.5 (g), p.109]{kiehl2001weil}, $R\alpha_*(\cE\otimes^L\cL_{\chi})=(R\alpha_*\cE)\otimes^L \cD_{\chi}$ in $D_c^b(\Alb(X))$. It follows  that \[H^k(X,\cE\otimes \cL_{\chi})=H^k(\Alb(X),(R\alpha_*\cE)\otimes^L \cD_{\chi})=H^{k-d-r(\alpha)}(\Alb(X), \cK\otimes^L \cD_{\chi}),\] whence (\ref{eq:SigmaXAlb}). Now Point \ref{it:Sigstri} follows from Fact \ref{ft:vanish} \ref{it:cls} and  Corollary \ref{cor:speccls}  \ref{it:Sigma>m}, and Point \ref{it:thin} follows from Corollary \ref{cor:speccls} \ref{it:Sigmathin}. 
\end{proof}
	\section{Generic vanishing result for manifolds in Fujiki class $\cC$}\label{sec:main}
In Section \ref{sec:Fujiki}, we recall the definition of Fujiki class $\cC$, the object of central interest in this note. Then we restrict mainly to algebraic varieties in Section \ref{sec:Moishezon}.
\subsection{Fujiki class $\cC$}\label{sec:Fujiki}

\begin{df}[Fujiki class $\cC$, {\cite[Def.~1]{ueno1980three}}]\label{df:Fujiki}
	A  compact complex manifold is called in Fujiki class $\cC$ if it is the meromorphic image of a compact Kähler manifold.
\end{df}
Every compact Kähler manifold is in Fujiki class $\cC$. 
The reason why Fujiki class $\cC$ is interesting is two-fold. For one thing,  this class is large enough in practice. For another,  in this class there is a Hodge theory with  unitary local systems as coefficients.
\begin{ft}[{\cite[Cor.~5.3]{timmerscheidt1987mixed}}]\label{ft:FujikiHodge} Let $X$ be a  complex manifold in Fujiki class $\cC$.  Then, for every unitary local system $\cE$ on $X$, $X$ is $\cE$-regular.\end{ft}

In particular, from Fact \ref{ft:FujikiHodge},  every manifold in Fujiki class $\cC$ is regular. As is explained in Section \ref{sec:Jac} and Section \ref{sec:Alb}, the Jacobian and Albanese of a complex manifold in Fujiki class $\cC$ behave well.

\medskip

\begin{thm}\label{thm:main}
	Let $X$ be an $n$-dimensional  complex manifold in Fujiki class $\cC$,  and let $\alpha:X\to \Alb(X)$ be the Albanese map associated with some base point. Let $E\to X$ be a flat unitary holomorphic vector bundle. Then for any   
	integers   $p,q\ge0$, one has:  
	\begin{enumerate}
		\item\label{it:locusana} The locus $S^{p,q}(X,E)$ is an analytic subset of $\Pic^0(X)$.
		\item\label{it:Serredual} $S^{n-p,n-q}(X,E^{\vee})=\{L\in \Pic^0(X)|L^{\vee}\in S^{p,q}(X,E)\}$.
		\item \label{it:Spq} If  $p+q<n-r(\alpha)$ or $p+q>n+r(\alpha)$, then $S^{p,q}(X,E)$ is  contained in  a strict (and  thin when $\Alb(X)$ is algebraic) subset of $\Pic^0(X)$. 
	\end{enumerate}
\end{thm}
\begin{proof}
\hfill	\begin{enumerate}
		\item The projection $p_2:X\times \Pic^0(X)\to \Pic^0(X)$ is a regular family in the sense of \cite[p.207]{grauert2012coherent}. Let $p_1:X\times \Pic^0(X)\to X$ be the other projection. Let $\cP$ be the universal line bundle  on $X\times \Pic^0(X)$ given by Corollary \ref{cor:Poincare}.
		Applying  the upper semi-continuity theorem (\cite[p.210]{grauert2012coherent}) to the vector bundle $\cP\otimes p_1^*\Omega_X^p$ and the regular family $p_2$, one gets that $S^{p,q}(E)$ is an analytic subset of $\Pic^0(X)$.

		\item By Serre duality (see, \textit{e.g.}, \cite[Prop.~4.1.15]{huybrechts2005complex}), for every $L\in \Pic(X)$, there is a perfect pairing \[H^q(X,\Omega_X^p\otimes_{O_X}L\otimes_{O_X}E)\times H^{n-q}(X,\Omega_X^{n-p}\otimes_{O_X}L^{\vee}\otimes_{O_X}E^{\vee})\to\C,\]  so $L\in S^{p,q}(X,E)$ if and only if $L^{\vee}\in S^{n-p,n-q}(X,E^{\vee})$. 
		\item By Theorem \ref{thm:unitRH},  there is   a unitary local system $\cE$ on $X$ such that $\cE\otimes_{\C}O_X$ is isomorphic to $E$. For each $\chi\in \car(X)$, let $L_{\chi}:=\cL_{\chi}\otimes_{\C}O_X$. Then 
		the isomorphism (\ref{eq:chartoPic}) of real Lie groups is given  by $\chi\mapsto L_{\chi}$. Moreover, the Hodge decomposition (\ref{eq:decompE}) for $\cE\otimes_{\C}\cL_{\chi}$ provided by Fact \ref{ft:FujikiHodge} is \[H^k(X,\cE\otimes_{\C}\cL_{\chi})=\oplus_{p+q=k}H^q(X,\Omega_X^p\otimes_{\C}\cE\otimes_{\C}\cL_{\chi})=H^q(X,\Omega_X^p\otimes_{O_X}E\otimes_{O_X}L_{\chi}).\] Therefore, under the isomorphism (\ref{eq:chartoPic}), one has \begin{equation}\label{eq:capT}\Sigma^k(X,\cE)\cap T(X)=\cup_{p+q=k}S^{p,q}(X,E).\end{equation}  The result follows from Theorem \ref{thm:premain}.
	\end{enumerate}
\end{proof}
\begin{rk}Theorem \ref{thm:main} \ref{it:Spq} extends Fact \ref{ft:KWvanish} from Kähler manifolds to Fujiki class $\cC$. As $\dim X-r(\alpha)\le \dim \alpha(X)$, the numerical hypothesis in  Theorem \ref{thm:main} is more restrictive than that in  Fact \ref{ft:GL}. An example from \cite[Remark, p.401]{green1987deformation} is reconsidered in the last paragraph of \cite[Sec. 3]{kramer2015vanish}, to show that the bound  $p+q<\dim X-r(\alpha)$  is optimal for Fact \ref{ft:KWvanish}. \end{rk}
\subsection{Moishezon manifolds}\label{sec:Moishezon}
Moishezon manifolds are  examples of manifolds in Fujiki class $\cC$. 
\begin{df}[Moishezon manifold, {\cite[Def.~2.2.12]{ma2007holomorphic}}] A connected compact complex manifold $X$ is called  Moishezon  if it has $\dim X$ algebraically independent meromorphic functions.\end{df} In fact, according to \cite[Thm.~2.2.16]{ma2007holomorphic}, for every Moishezon manifold  $X$, there is a proper modification $\pi:X'\to X$ with $X'$ a projective manifold. In particular, $X$ is the meromorphic image of a projective manifold, hence in Fujiki class $\cC$. Conversely, a connected compact complex manifold that is the meromorphic image of a projective manifold is Moishezon by the proof of   \cite[Cor.~12.12]{voisin2002hodge}. For more references, see  \cite[Sec.~1]{jia2022moishezon}.

The intersection of the two subclasses, Kähler and Moishezon, is exactly the class of projective manifolds. More precisely, Moishezon's theorem (see, \textit{e.g.}, \cite[Thm.~12.13]{voisin2002hodge}) asserts that a Moishezon manifold  is Kähler if and only if it is projective. A Moishezon manifold may not be homotopy equivalent to a Kähler manifold (\cite[Thm.~1]{oguiso1994two}). Kodaira-Spencer stability theorem (see, \textit{e.g.}, \cite[Thm.~9.1]{voisin2002hodge}) shows that small deformations of a Kähler manifold are Kähler. Similarly, small deformations of a regular manifold are regular (\cite[Cor.~3.7]{angella2013partial}). By contrast, there is a small deformation of a  Moishezon manifold that is not in Fujiki class $\cC$ (\cite[Sec.~0]{campana1991classc}). In particular, there exists a regular manifold that is not in Fujiki class $\cC$.

Moishezon manifolds are abundant. For example, for every smooth proper algebraic variety $X/\C$, its analytification $X^{\an}$ is a Moishezon manifold (\cite[p.442]{hartshorne2013algebraic}). Hironaka (\cite{hironaka1960theory}, see also \cite[p.443]{hartshorne2013algebraic}) gives examples of Moishezon manifolds that are not algebraic, 
and  smooth proper algebraic varieties that are not projective. The situation is depicted below. Every inclusion in  this graph is strict.
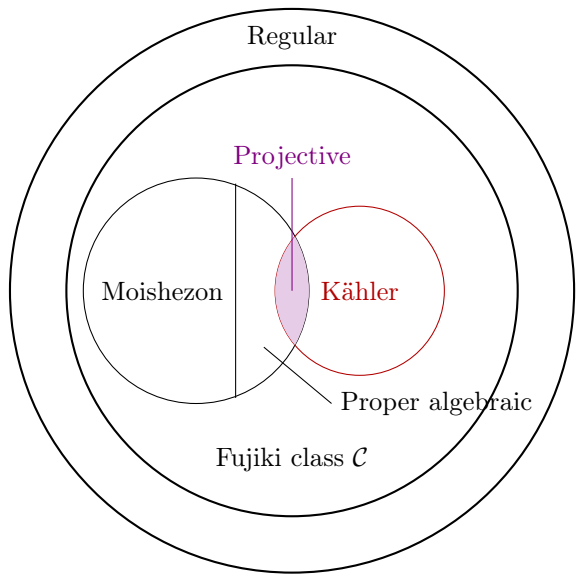
\begin{center}
	\begin{tikzpicture}[scale=0.75]
		\draw[color=black,thick] (0,0) circle (5cm);
		\node[color=black] at (0cm,4.5cm) {Regular};
		\draw[color=black,thick] (0,0) circle (4cm);
		\node[color=black] at (0cm,-3cm) {Fujiki class $\cC$};
		\draw[color=black] (-1.7cm,0cm) circle (2cm);
		\node[color=black] at (-2.3cm,0cm) {Moishezon};
		\draw[color=black] (-1,1.9cm)--(-1,-1.9) ;
		\draw[color=red!70!black] (1.2,0cm) circle (1.5cm) node {Kähler};
	\begin{scope}
			\clip (1.2cm,0cm) circle (1.5cm);
			\fill[fill=violet!20] (-1.7cm,0cm) circle (2cm);
		\end{scope}
		\draw[color=violet] (0cm,0cm)--(0,2) node[above] {Projective};
		\draw[color=black] (-0.5,-1)--(0.7,-2)node[right] {Proper algebraic};
	\end{tikzpicture}
\end{center}

We need   Proposition \ref{pp:AlbMoi} on the algebraicity of Picard torus and Albanese torus to compare them with the Picard variety and Jacobian variety of an algebraic variety.
\begin{pp}\label{pp:AlbMoi}
	If $X$ is a  Moishezon manifold, then $\Alb(X)$ and $\Pic^0(X)$  are complex abelian varieties dual to each other.
\end{pp}

\begin{proof}
	By \cite[Thm.~2.2.16]{ma2007holomorphic},	$X$ admits a proper modification $\pi:X'\to X$ with $X'$ a projective manifold.	By \cite[Prop.~7.16]{voisin2002hodge}, the Jacobian $\Pic^0(X')$ is projective. From Proposition \ref{pp:Alb} \ref{it:AlbPic0}, the torus $\Alb(X')$ is  dual to $\Pic^0(X')$,  so $\Alb(X')$ is algebraic. By \cite[Prop.~9.12, p.107]{ueno2006classification}, the morphism  $\pi_*:\Alb(X')\to \Alb(X)$ given by Proposition \ref{pp:Alb} \ref{it:Albcanonical} is an isomorphism. 
\end{proof}
\begin{rk}
By \cite[ p.70]{birkenhake2004complex}, the analytic dual torus of a complex abelian variety is an abelian variety. Moreover, by \cite[p.86]{mumford1974abelian}, the (algebraic) dual abelian variety (defined in \cite[p.78]{mumford1974abelian}) of a complex abelian variety coincides with its analytic dual torus, so we do not distinguish the two duals in this case.
\end{rk}

	From now on, let $X/\C$ be a smooth proper algebraic variety of dimension $n$ with a base point $x_0\in X(\C)$, and let $\Sch/\C$ (resp. $\Sets$) be the  category of $\C$-schemes (resp. sets). The fppf-sheaf associated to the functor \[P_{X/\C}:(\Sch/S)^{\op}\to \Sets,\quad T\mapsto \Pic(X\times_{\C}T)\] is called the \emph{relative Picard functor} of $X$  (\cite[Def.~2, p.201]{bosch2012neron}). From \cite[p.211, p.231 and p.233]{bosch2012neron}, the relative Picard functor of $X$ is represented by a smooth group scheme $\Pic_{X/\C}$ over $\C$. In particular, the group $\Pic_{X/\C}(\C)=\Pic(X)$. By \cite[Thm.~3, p.232]{bosch2012neron}, the identity component  $\Pic^0_{X/\C}$ of $\Pic_{X/\C}$ is proper over $\C$, hence a complex  abelian variety called the \emph{Picard variety} of $X$. 
	
	From \cite[Thm.~5]{serre1958morphismes}, there is an abelian variety $\Alb(X)/\C$ with a  $\C$-morphism  $\alpha_{X,x_0}:(X,x_0)\to (\Alb(X),0)$ of pointed varieties satisfying the following universal property:\footnote{similar to that stated in Proposition \ref{pp:Alb} \ref{it:Albuniv}} every $\C$-morphism of pointed varieties $(X,x_0)\to (A,0)$ with $A/\C$ an abelian variety factors uniquely through a morphism of abelian varieties $\Alb(X)\to A$. Such morphism $\alpha_{x_0}$ is unique up to a unique isomorphism. We call $\Alb(X)$ the \emph{algebraic Albanese variety} of $X$ and $\alpha_{X,x_0}:(X,x_0)\to (\Alb(X),0)$ the \emph{algebraic Albanese morphism} corresponding to $x_0$.
	
For every $O_X$-module $F$, let $F^{\an}$ be the corresponding $O_{X^{\an}}$-module defined in \cite[Exp. XII, 1.3]{SGA1}. Hence a functor \[\Mod(O_X)\to \Mod(O_{X^{\an}}),\quad F\mapsto F^{\an}.\]
	By Serre's GAGA \cite[Exp. XII, Thm.~4.4]{SGA1}, the natural group morphism \[\Pic(X)\to \Pic(X^{\an}),\quad L\mapsto L^{\an}\] is an isomorphism. Corollary \ref{cor:Albalg} \ref{it:Picalgan} of GAGA type compares the algebraic Picard variety and the analytic Jacobian. Once again, it is well-known, but a proof is given for the lack of reference. 
\begin{cor}\label{cor:Albalg}
\hfill\begin{enumerate}
		\item\label{it:A=Alb} The analytification of  $\Alb(X)$ (resp. $\alpha_{X,x_0}:X\to \Alb(X)$ )  is $\Alb(X^{\an})$  (resp. $\alpha_{X^{\an},x_0}:X^{\an}\to \Alb(X^{\an})$).
		\item\label{it:Picalgan} The analytification of  $\Pic^0_{X/\C}$ is $\Pic^0(X^{\an})$.
	\end{enumerate}
\end{cor}
\begin{proof}
\hfill\begin{enumerate}\item	Since $X^{\an}$ is a Moishezon manifold, by Proposition \ref{pp:AlbMoi}, its Albanese torus $\Alb(X^{\an})$ is projective. By Chow's theorem \cite[Cor.~A.4]{birkenhake2004complex}, the map $\alpha_{X^{\an}, x_0}$ is algebraic. By Proposition \ref{pp:Alb} \ref{it:Albuniv},  every algebraic morphism $(X,x_0)\to (A,0)$ to an abelian variety $A/\C$ factors   uniquely through an analytic (hence algebraic by Chow's theorem again) morphism of complex tori $\Alb(X^{\an})\to A^{\an}$. The result follows.
	
\item By \cite[Prop.~A.6]{mochizuki2012topics}, the (algebraic) dual abelian variety of $\Pic^0_{X/\C}$ is $\Alb(X)$. By Proposition \ref{pp:Alb} \ref{it:AlbPic0}, $\Pic^0(X^{\an})$ is the (analytic) dual torus of $\Alb(X^{\an})=\Alb(X)^{\an}$, so $\Pic^0(X^{\an})$  is the analytification of $\Pic^0_{X/\C}$.\end{enumerate}
\end{proof}
Identifying $\Pic^0_{X/\C}$ with $\Pic^0(X^{\an})$ \textit{via} Corollary \ref{cor:Albalg} \ref{it:Picalgan}, one can define thin subsets of $\Pic^0_{X/\C}$. Define the defect of semismallness of a proper morphism $f:M\to N$ between complex algebraic varieties by $r(f)=r(f^{\an})$. With this terminology, we get   the following generic vanishing result for smooth proper algebraic varieties.
\begin{cor}\label{cor:agmain}
	 Let $\cE$ be a \emph{unitary} local system on $X^{\an}$, and let $E=\cE\otimes_{\C}O_{X^{\an}}$ be the corresponding holomorphic vector bundle. Then for any integers  $p,q\ge0$ with $p+q> n+r(\alpha)$ or $p+q<n-r(\alpha)$, the locus $S^{p,q}(X^{\an},E)$ is contained in a thin (and arithmetic when $\cE$ is semisimple of geometric origin in $D_c^b(X^{\an})$) subset of $\Pic^0_{X/\C}$.
\end{cor}
\begin{proof}	
By Corollary \ref{cor:Albalg} \ref{it:A=Alb},   the analytification $\alpha^{\an}_{X,x_0}:X^{\an}\to \Alb(X)^{\an}$ coincides with  $\alpha_{X^{\an},x_0}:X^{\an}\to \Alb(X^{\an})$, and by definition, $r(\alpha)=r(\alpha^{\an})$.  From Theorem \ref{thm:main} \ref{it:Spq},  the locus $S^{p,q}(X^{\an},E)$ is contained in a thin subset of $\Pic^0(X)$.
	
	What remains to show is the assertion in the parentheses.	Assume that $\cE$ is semisimple of geometric origin.    By the decomposition theorem \cite[Thm.~6.2.5]{beilinson2018faisceaux},   $\cK:=R\alpha_*\cE[n+r(\alpha)]$  is semisimple of geometric origin in $D_c^b(\Alb(X^{\an}))$. By Theorem \ref{thm:main} \ref{it:Serredual}, one may assume that $p+q>n+r(\alpha)$, so that \[S^{p,q}(X^{\an},E)\subset \Sigma^{p+q}(X^{\an},\cE)\cap T(X)\subset \Sigma^{>0}(\Alb(X),\cK),\] where the first inclusion follows from (\ref{eq:capT}) and the second from (\ref{eq:SigmaXAlb}).  From Corollary \ref{cor:speccls} \ref{it:Sigmathin},   $\Sigma^{>0}(\Alb(X),\cK)$ is contained in an arithmetic thin subset of $\Pic^0_{X/\C}$. 
\end{proof}
\begin{rk}By Chow's theorem, every analytic subset of $X^{\an}$ is algebraic. Therefore, $D_c^b(X^{\an})$ coincides with $D_c^b(X(\C),\C)$ defined in \cite[p.66]{beilinson2018faisceaux} using \emph{algebraic} Whitney stratifications.\end{rk}

		\bibliography{genericvanishing.bib}
	\bibliographystyle{alpha}\end{document}